\documentclass{amsart}
\usepackage{graphicx,hyperref}
\usepackage{verbatim}
\usepackage{tikz}
\usepackage{url}

\DeclareMathOperator{\Var}{Var}
\DeclareMathOperator{\vectorize}{vec}

\DeclareMathOperator{\Cov}{Cov}

\DeclareMathOperator{\tr}{tr}
\DeclareMathOperator{\diag}{diag}
\newcommand{\vect}[1]{\boldsymbol{#1}}
\renewcommand{\vec}{\vect}

\usepackage{caption}
\newtheorem{theorem}{Theorem}[section]
\newtheorem{lemma}{Lemma}[section]
\newtheorem{step}{Step}[section]
\newtheorem{corollary}{Corollary}[section]
\newtheorem{proposition}{Proposition}[section]

\theoremstyle{definition}
\newtheorem{definition}{Definition}[section]

\theoremstyle{remark}
\newtheorem{remark}{Remark}[section]

\title[Optimal risk-sharing in Decentralized Insurance]{Optimal Risk-Sharing Rules in Network-based Decentralized Insurance}

\author[H.N. Fogarty]{Heather N. Fogarty}
\email{fogartyh@oregonstate.edu}
\address{Department of Mathematics, Oregon State University}

\author[S-H. Loke]{Sooie-Hoe Loke}
\email{sooiehoe.loke@mtsu.edu}
\address{Department of Mathematical Sciences, Middle Tennessee State University}

\author[N.F. Marshall]{Nicholas F. Marshall}
\email{marsnich@oregonstate.edu}
\address{Department of Mathematics, Oregon State University}

\author[E.A. Thomann]{Enrique A. Thomann}
\email{thomann@math.oregonstate.edu}
\address{Department of Mathematics, Oregon State University}
\begin{document}
\begin{abstract}
This paper studies decentralized risk-sharing on networks. In particular, we consider a model where agents are nodes in a given network structure. Agents directly connected by edges in the network are referred to as friends. We study actuarially fair risk-sharing under the assumption that only friends can share risk, and we characterize the optimal signed linear risk-sharing rule in this network setting. Subsequently, we consider a special case of this model where all the friends of an agent take on an equal share of the agent's risk, and establish a connection to the graph Laplacian. Our results are illustrated with several examples.
\end{abstract}

\maketitle

\section{Introduction}

Risk-sharing has been extensively studied in actuarial science.
Early works on the mathematical analysis of risk-sharing include work in the 1960s by Borch \cite{borch1960attempt,borch1962equilibrium,borch1968general}, who coined the term non-olet risk-sharing, where a central authority pools the agents' risks and then redistributes them without taking into account the origins of the risks in the pool. From a mathematical perspective, this means that while the risk-sharing rule of each agent may depend on individual parameters, it is only a function of the aggregate losses.

Recently, decentralized insurance models have attracted considerable attention in the actuarial research community. In these models, agents share risk among each other with limited or no role for a central authority; see \cite{feng2023decentralized,Feng2024} for a systematic mathematical treatment of decentralized insurance. While technological, economic, and social developments have created a renewed interest in decentralized insurance models, there already exist historical examples of such risk-sharing schemes. A prominent example is Takaful, which is an Islamic-compliant form of insurance based on mutual assistance, where participants contribute to a common pool to cover losses, rather than transferring risk to an insurer; see \cite{malik2019introduction}. Decentralized insurance models have also found a variety of modern applications, including cyber insurance contracts \cite{fahrenwaldt2018pricing}, cooperative insurance \cite{clemente2024risk}, and, recently, have been used by governments to manage catastrophic risks, especially related to climate change and extreme weather events \cite{bollmann2019international}.

In this paper, we are specifically interested in peer-to-peer (P2P) insurance, which is a model where agents directly exchange risk. The mathematical foundations of P2P insurance are an active area of research, with common topics of examination including optimization conditions \cite{abdikerimova2024multiperiod,denuit2012convex,denuit2025comonotonicity, yang2024optimality}, development of new risk-sharing models \cite{boonen2025peer,denuit2022,feng2023peer,levantesi2022mutual}, and analysis of existing risk-sharing models \cite{abdikerimova2022peer,denuit2020investing,denuit2021risk, denuit2021risk2}.  
A key difference between centralized insurance models and P2P models is that P2P insurance enables arrangements where participants do not share risk with every other agent, which can be modeled using a network whose nodes are agents and whose edges represent potential risk-sharing relationships. From the network theory perspective, non-olet risk-sharing corresponds to a star graph, where the central node is the central authority, while unrestricted P2P risk-sharing corresponds to a complete graph,  
 see Figure \ref{figgraphs}. 

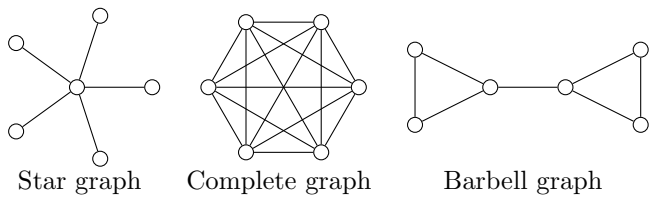
\begin{figure}[h!]
\centering
\begin{tabular}{ccc}
\raisebox{-.5\height}{
\begin{tikzpicture}[scale=1, every node/.style={circle, draw, fill=white, inner sep=2pt}]
  \node (c) at (0,0) {};

  \foreach \i in {0,...,4} {
    \node (v\i) at ({cos(72*\i)}, {sin(72*\i)}) {};
    \draw (c) -- (v\i);
  }
  
\end{tikzpicture}}
&
\raisebox{-.5\height}{
\begin{tikzpicture}[scale=1, every node/.style={circle, draw, fill=white, inner sep=2pt}]
  \foreach \i in {1,...,6} {
    \node (v\i) at ({cos(60*\i)}, {sin(60*\i)}) {};
  }

  \foreach \i in {1,...,6} {
    \foreach \j in {1,...,6} {
      \ifnum\i<\j
        \draw (v\i) -- (v\j);
      \fi
    }
  }
\end{tikzpicture}}
&
\raisebox{-.5\height}{
\begin{tikzpicture}[scale=1, every node/.style={circle, draw, fill=white, inner sep=2pt}]
  \node (a1) at (0,0) {};
  \node (a2) at (0,1) {};
  \node (a3) at (1,0.5) {};
  \draw (a1) -- (a2) -- (a3) -- (a1);

  \node (b1) at (3,0) {};
  \node (b2) at (3,1) {};
  \node (b3) at (2,0.5) {};
  \draw (b1) -- (b2) -- (b3) -- (b1);

  \draw (a3) -- (b3);
\end{tikzpicture}} \\
Star graph & Complete graph & Barbell graph
\end{tabular}
\caption{Many works on  P2P insurance either perform non-olet risk pooling (represented by the star graph) or unrestricted risk-sharing (represented by the complete graph). In this work, we consider networks with general structures such as the Barbell graph.} \label{figgraphs}
\end{figure}

While a substantial body of work on P2P insurance has emerged, particularly over the past decade, the theoretical foundations of P2P risk-sharing on networks remain underdeveloped; notable examples include
\cite{charpentier2021}, which examines decentralized insurance on a random network,
and \cite{charpentier2025linear} that considers optimal row and column stochastic risk-sharing rules on networks.

In this paper, we study optimal risk-sharing on networks inspired by \cite{feng2023peer}, which considers optimal unrestricted (signed) linear risk-sharing, which corresponds to risk-sharing on a complete graph. 
 In particular, \cite{feng2023peer} characterizes the linear actuarially fair risk-sharing rule that is optimal in the sense that it minimizes the sum of variances of each agent's loss after risk-sharing. Subsequently, \cite{yang2024optimality} proves that, among all risk-sharing rules, the optimal risk-sharing rule is an affine function of the 
residual risks (formed by subtracting the mean from each loss random variable).
  Recently, \cite{niakh2025fixed} shows that actuarially fair Pareto-optimal risk-sharing rules are in one-to-one correspondence with the fixed points of a specific function. 
  Some of the aforementioned works have been extended to multi-period risk-sharing \cite{abdikerimova2024multiperiod} and a continuous time setting
\cite{boonen2025robust}.

\subsection{Main contributions}
In this paper, we consider a variance minimization problem for a network-based, actuarially fair, linear risk-sharing rule, which can be applied to any connected network. In this context, the main contributions of this paper are as follows: (1) We characterize the optimal (signed) allocation, in Theorem \ref{MainResults}, for connected networks, extending results in \cite{feng2023peer} which are only valid for complete graphs; (2) We establish a novel connection to the graph Laplacian in the special case that risk is proportionally shared among agents with a common node, as demonstrated in  Theorem \ref{thmoptshare}; and (3) We obtain necessary and sufficient conditions for the nonnegativity of risk allocation for complete graphs (Proposition \ref{lem:gencase}), and for risk-sharing rules modeled using the graph Laplacian, (Lemma \ref{FriendsEqualPos}). 

\subsection{Preliminaries} \label{sec:prelmin}
In this section, we give a precise mathematical definition of risk-sharing. 
Let $\vec X = (X_1,\ldots, X_n)^\top$ be a nonnegative random vector whose $i$-th entry $X_i$ represents the loss of agent $i$. Let $\vec\mu  = \mathbb{E} [\vec X]$ denote the mean of $\vec{X}$ and 
$$\vec\Sigma=
\Var(\vec X) := \mathbb{E} \left[ (\vec X - \mathbb{E} [\vec X]) (\vec X - \mathbb{E} [\vec X])^\top \right],
$$
denote the covariance matrix of $\vec{X}$. Throughout this paper, we assume that $\vec{\Sigma}$ is  positive definite and therefore invertible. 
The assumption that the covariance matrix $\vec \Sigma$ is positive definite implies that each entry of $\vec X$ has strictly positive variance. Since $\vec{X}$ is nonnegative by assumption, the fact that its entries have positive variance implies that each entry of $\vec{X}$ is positive with nonzero probability. Thus, the mean vector $\vec \mu$ has strictly positive entries.

A \textit{risk-sharing rule} $H$ is a function $H: \mathbb{R}^n \rightarrow \mathbb{R}^n$ that satisfies the following full-allocation property.

\begin{definition}
\label{fullallocation}
A function $H: \mathbb{R}^n\to \mathbb{R}^n$ satisfies the full-allocation property if
$$
\sum_{i=1}^n H_i(\vec{X}) = \sum_{i=1}^n X_i,
$$
almost surely, where $H_i(\vec{X})$ denotes the $i$-th entry of $H(\vec{X})$.
\end{definition}

An example of a risk-sharing rule is the linear rule $H : \mathbb{R}^n \to \mathbb{R}^n$ by
\begin{equation} \label{eq:linear}
H(\vec{X}) = \vec{AX},
\end{equation}
where $\vec{A}=(a_{ij}) \in \mathbb{R}^{n \times n}$ is any matrix whose columns sum to $1$, that is, $\vec{1}^\top \vec{A} = \vec{1}^\top$, where $\vec{1}$ is a column vector of ones.

Note that some papers require risk-sharing functions to be nonnegative, which corresponds to the property that agents cannot profit from the loss of another agent. Under such an assumption, the matrix $\vec{A}$ in \eqref{eq:linear} would be restricted to a column stochastic matrix. In this paper, 
we start by making no nonnegativity assumption,
which allows us to characterize risk-sharing allocations that take negative values. Practically speaking, a risk-sharing scheme involving signed exchanges is analogous to financial models that allow agents to buy and short-sell financial instruments. 
Conditions for achieving nonnegative risk-sharing rules are discussed in \S \ref{sec:nonneg}.

As customary, we also impose that the risk-sharing rules are \emph{actuarially fair} according to the following definition.

\begin{definition} 
\label{defactfair}
A risk-sharing rule $H: \mathbb{R}^n\to\mathbb{R}^n$ is actuarially fair if
$$
\mathbb{E} [H(\vec X)] = \mathbb{E}[ \vec X ].
$$
\end{definition}

\subsection{Prior work} \label{sec:priorwork}
Recall that $\vec\mu=\mathbb{E}[\vec X]$ is the mean vector and $\vec\Sigma = \Var(\vec X)$ is the covariance matrix of $\vec X$. Feng, Liu, and Taylor \cite{feng2023peer} consider the following optimization problem  for a linear risk-sharing rule $H(\vec{X}) = \vec{A} \vec{X}$:
\begin{equation} \label{eq:optproblemlin}
\begin{cases}
\text{minimize} &  \frac{1}{2} \tr( \vec{A} \vec{\Sigma} \vec{A}^\top)\\
\text{subject to} & \vec{1}^\top \vec A=\vec{1}^\top, \quad
\vec{A} \vec{\mu}= \vec\mu,
\end{cases}
\end{equation}
where the optimization is taken over $\vec{A} \in \mathbb{R}^{n \times n}$.
The objective function in this optimization problem is half the sum of the variances of the agents' losses after risk-sharing, as  given by
$$
\frac{1}{2} \tr( \vec{A} \vec{\Sigma} \vec{A}^\top) =
\frac{1}{2}\tr\Var(\vec{AX}) = \frac{1}{2}\sum_{i=1}^n\Var(H_i(\vec X)).
$$
The first constraint $\vec1^\top\vec A = \vec1^\top$ enforces full-allocation (Definition \ref{fullallocation}), and the second constraint  $\vec{A} \vec{\mu} =\vec\mu$ ensures actuarial fairness (Definition \ref{defactfair}).
We emphasize again that $H(\vec{X}) = \vec{A} \vec{X}$ is allowed to be a signed risk-sharing rule  and  $\vec{A}$ may have some negative entries,  see \S \ref{examplenegative}
for an example where negative entries arise.  The following result characterizes the solution to the optimization problem \eqref{eq:optproblemlin}.

\begin{theorem}[Feng, Liu, Taylor \cite{feng2023peer}]
\label{thm:feng}
The optimization problem \eqref{eq:optproblemlin} has a unique solution 
\begin{equation} \label{fang2023eq}
 \vec{A}_* =\frac{1}{n} \vec{1} \vec{1}^\top +\frac{1}{a}\left(\vec I- \frac{1}{n} \vec{1} \vec{1}^\top\right)\vec{\mu \mu}^\top \vec\Sigma^{-1},
\end{equation}
where $a = \vec\mu^\top \vec\Sigma^{-1}\vec\mu$.
\end{theorem}

The proof of this result is based on Lagrange multipliers, see \cite{feng2023peer} for details.  In this paper, we consider a generalization of the optimization problem \eqref{eq:optproblemlin} that allows for network-based constraints.

\section{Main results}
\subsection{Only friends share risk} \label{sec:mainresult}
As above, let $\vec{X}$ be an $n$-dimensional loss vector with mean $\vec{\mu}$ and covariance $\vec{\Sigma}$. Assume that there is an underlying simple undirected graph $G = (V,E)$ whose vertices $V = \{1,\ldots,n\}$ correspond to the agents, and whose edge set $E$, which consists of unordered pairs of distinct vertices $\{i,j\}$, represents friendships between the agents.  Let $d_i = \#\{ j \in V : \{i,j\} \in E \}$ denote the degree of vertex $i$. Let
$$
\overline{E} = \{ \{i,j\} : i,j \in V,\, i \ne j,\, \{i,j\} \not \in E\}
$$
denote the edge set of the complement graph of $G$. That is, $\overline{E}$ consists of all unordered pairs of distinct vertices that are not edges of $G$. We emphasize that both $E$ and $\overline{E}$ consist of unordered pairs of \emph{distinct} vertices, so $\{i,i\} \not \in E$ and $\{i,i\} \not \in \overline{E}$ for all $i$.

Throughout this paper, we assume that the number of agents $n \ge 2$ to avoid a trivial situation.  In this paper, we consider the following network-based constraint.

\begin{definition}[Only friends share risk]  \label{assumption1}
Let $H: \mathbb{R}^n \to \mathbb{R}^n$ be a risk-sharing rule. We say that $H$ has the property that only friends share risk if, for each fixed $i \in \{1,\ldots,n\}$, we have
$$
H_i(\vec x) = f_i\left(x_i,x_{j_1},\ldots,x_{j_{d_i}} \right),
$$
for some function $f_i : \mathbb{R}^{d_i+1} \to \mathbb{R}$, where $\vec{x} = (x_1,\ldots,x_n)$ and $j_1,\ldots,j_{d_i}$ is an enumeration of the vertices $\{j \in V : \{i,j\} \in E \}$.
\end{definition}

We consider the following optimization problem for a linear risk-sharing rule $H(\vec{X}) = \vec{A} \vec{X}$  where only friends share risk:
\begin{equation}
\label{opteq}
\begin{cases}
\text{minimize} & \frac{1}{2}\tr(\vec{A} \vec{\Sigma} \vec{A}^\top) \\
\text{subject to} &  \vec{1}^\top \vec{A} = \vec{1}^\top, \quad  \vec{A} \vec{\mu}  = \vec{\mu} , \quad 
a_{ij} = 0 \text{ whenever } \{i,j\} \in \overline{E},
\end{cases}
\end{equation}
where the optimization is taken over the matrix $\vec{A} = (a_{ij}) \in \mathbb{R}^{n \times n}$. As in 
\S \ref{sec:priorwork}, the objective function of the optimization is half the sum of the variances of the losses after risk-sharing, the constraint $\vec{1}^\top \vec{A} = \vec{1}^\top$ enforces full allocation (Definition 
\ref{fullallocation}), and the constraint $\vec{A} \vec{\mu}= \vec{\mu}$ enforces actuarial fairness (Definition \ref{defactfair}). The final constraint,
 $a_{ij} =0$ whenever $\{i,j\} \in \overline{E}$, enforces that only friends share risk (Definition \ref{assumption1}). The following result characterizes the solution to this optimization problem.

\begin{theorem}[Only friends share risk]\label{MainResults}
The optimization problem \eqref{opteq} has a unique solution
\begin{equation}
\label{optimalrestricted}
 \vec A_* = \frac{1}{n} \vec{1 1}^\top  + \left(\vec I - \frac{1}{n} \vec{1 1}^\top \right)
\left( \frac{1}{a} \vec{\mu \mu}^\top + \vec{\Gamma} \left(
\frac{1}{a} \vec\Sigma^{-1}\vec{\mu \mu}^\top
 - \vec I \right) \right) \vec\Sigma^{-1},
\end{equation} 
where $a = \vec\mu^\top \vec\Sigma^{-1}\vec\mu$ and $\vec{\Gamma} = (\gamma_{ij}) \in \mathbb{R}^{n \times n}$ satisfies $\gamma_{ij} = 0$ when $i=j$ or $\{i,j\} \in E$, and the other entries $\gamma_{ij}$ 
are chosen so that the equations
\begin{equation} \label{eq:linearsystem1}
 \left( 
  \frac{1}{n} \vec{1 1}^\top   +
 \left(\vec I - \frac{1}{n} \vec{1 1}^\top \right)
 \left(
   \frac{1}{a} \vec{\mu \mu}^\top +
 \vec{\Gamma}
\left(\frac{1}{a} \vec\Sigma^{-1}\vec{\mu \mu}^\top
 - \vec I \right)\right)   \vec\Sigma^{-1} 
 \right)_{ij}
 = 0
 \end{equation}
are satisfied for all $i,j$ such that $\{i,j\} \in \overline{E}$. Such a matrix $\vec{\Gamma}$ always exists but need not be unique.

\end{theorem}

The proof of Theorem \ref{MainResults} is given in \S \ref{sec:proofmainresult}. Several remarks are in order.

\begin{remark}
We emphasize that the matrix $\vec{\Gamma}$ in the statement of
 Theorem \ref{MainResults}
need not be unique if the linear system
of equations \eqref{eq:linearsystem1} is rank deficient. However, a solution is guaranteed to exist, and substituting any solution $\vec{\Gamma}$ into \eqref{optimalrestricted} will result in the same unique matrix $\vec{A}_*$. 
\end{remark}

\begin{remark} 
The linear system of equations defined by \eqref{eq:linearsystem1} states that the
$(i,j)$-th entry of a matrix equals $0$ for all $i,j$ such that $\{i,j\} \in \overline{E}$. Since $\{i,j\}$ 
denotes an unordered pair, this means that both the $(i,j)$-th and $(j,i)$-th entries of the final $\vec{A}_*$ equal $0$ whenever $\{i,j\} \in \overline{E}$.
Similarly, the collection of unknowns in the system of equations includes both $\gamma_{ij}$ and $\gamma_{ji}$ whenever $\{i,j\} \in \overline{E}$.

So for example, if $\overline{E} = \{\{2,4\}\}$, then \eqref{eq:linearsystem1} is a linear system consisting of two equations corresponding to $(i,j) = (2,4)$ and $(i,j) = (4,2)$ with two unknowns $\gamma_{2,4}$ and $\gamma_{4,2}$.  
\end{remark}

\begin{remark}
\label{extension}
Note that Theorem \ref{MainResults} is an extension of Theorem \ref{thm:feng}.  Indeed, in the case of a complete graph, \eqref{eq:linearsystem1} is a vacuous requirement and $\vec{\Gamma}=\vec{0}$.
Thus $\vec{A}_*$ defined in \eqref{optimalrestricted} agrees with $\vec{A}_*$ defined in \eqref{fang2023eq}.
\end{remark}

\begin{remark}
The optimization problems considered in this paper are formulated from the perspective of a social planner seeking to reduce the aggregate variance of the post-sharing losses subject to actuarial fairness and network constraints. While this objective reduces the overall risk in the system, it does not necessarily reduce the risk of every individual agent. Thus, our framework does not impose individual rationality constraints. Studying network risk sharing rules that explicitly account for individual risk and incentives to join is an interesting direction for future work; see Section \ref{sec:discussion}.
\end{remark}

 Additionally, see Remark \ref{rmk:computingAstart} for a discussion about the computation of $\vec{\Gamma}$.

\subsection{Examples of Theorem \ref{MainResults}} \label{sec:examples1}
In the following, we illustrate Theorem \ref{MainResults} through several examples.
In each example, for a given mean vector $\vec{\mu}$, covariance matrix $\vec{\Sigma}$, and graph $G= (V,E)$, we use Theorem \ref{MainResults} to determine the unique $\vec{A}_*$ that solves the optimization problem \eqref{opteq}, and report the value of the objective function for the optimal $\vec{A}_*$.

\subsubsection{Complete graph} \label{sec:completegraph}
To establish a basis of comparison, we start by considering a complete graph with uncorrelated unit mean and variance losses. Since the network is fully connected, the optimization problem  
\eqref{opteq} is equivalent to the optimization problem 
\eqref{eq:optproblemlin} considered in \cite{feng2023peer}. Here, we have
$$
\vec{\mu} = \vec{1} \quad \vec{\Sigma} = \vec{I} \quad
\raisebox{-.45\height}{\begin{tikzpicture}[scale=.75, every node/.style={circle, draw, fill=white, inner sep=2pt}]

\node (1) at (0,1) {1};
\node (2) at (2,1) {2};
\node (3) at (2,-1) {3};
\node (4) at (0,-1) {4};

\draw (1) -- (2) -- (3) -- (4) -- (1);

\draw (3) -- (1);
\draw (2) -- (4);

\end{tikzpicture}}
\quad
\vec{A}_*=\begin{bmatrix}
\frac{1}{4}&\frac{1}{4}&\frac{1}{4}&\frac{1}{4}\\[2pt]
    \frac{1}{4}&\frac{1}{4}&\frac{1}{4}& \frac{1}{4}\\[2pt]
    \frac{1}{4}&\frac{1}{4}&\frac{1}{4}&\frac{1}{4}\\[2pt]
    \frac{1}{4}&\frac{1}{4}&\frac{1}{4}&\frac{1}{4}
\end{bmatrix},
$$
and $\frac{1}{2}\tr(\vec{A}_* \vec{\Sigma} \vec{A}_*^\top) = \frac{1}{2} = 0.5$.

\subsubsection{Complete graph with one edge removed}
\label{sec:oneedgeremoved}
We modify the previous example by removing the edge $\{2,4\}$ so that agents $2$ and $4$ are not allowed to share risk:
$$
\vec{\mu} = \vec{1} \quad \vec{\Sigma} = \vec{I} \quad
\raisebox{-.45\height}{\begin{tikzpicture}[scale=.75, every node/.style={circle, draw, fill=white, inner sep=2pt}]

\node (1) at (0,1) {1};
\node (2) at (2,1) {2};
\node (3) at (2,-1) {3};
\node (4) at (0,-1) {4};

\draw (1) -- (2) -- (3) -- (4) -- (1);

\draw (3) -- (1);

\end{tikzpicture}}
\quad
\vec{A}_* =\begin{bmatrix}
    \frac{1}{5}&\frac{3}{10}&\frac{1}{5}&\frac{3}{10}\\[2pt]
\frac{3}{10}&\frac{2}{5}&\frac{3}{10}&0\\[2pt]
    \frac{1}{5}&\frac{3}{10}&\frac{1}{5}&\frac{3}{10}\\[2pt]
    \frac{3}{10}&0&\frac{3}{10}&\frac{2}{5}
\end{bmatrix},
$$
and
$\frac{1}{2}\tr(\vec{A}_*\vec \Sigma \vec{A}_*^\top) = \frac{3}{5} = 0.6$. The entries $(2,4)$ and $(4,2)$ of $\vec{A}_*$ are equal to zero since agents $2$ and $4$ cannot exchange risk, and the value of the objective function has slightly increased due to the additional restriction.

\subsubsection{Positive correlated losses} \label{sec:poscorr}
We further modify the previous example by changing the covariance matrix $\vec{\Sigma}$ such that some losses are positively correlated:
$$
\vec{\mu} = \vec{1} \quad 
\vec\Sigma =
\begin{bmatrix}
    1&\frac{1}{3}&0&\frac{1}{3}\\[2pt]
    \frac{1}{3}&1&\frac{1}{3}&0\\[2pt]
    0&\frac{1}{3}&1&\frac{1}{3}\\[2pt]
    \frac{1}{3}&0&\frac{1}{3}&1
\end{bmatrix}
 \quad
\raisebox{-.45\height}{\begin{tikzpicture}[scale=.75, every node/.style={circle, draw, fill=white, inner sep=2pt}]

\node (1) at (0,1) {1};
\node (2) at (2,1) {2};
\node (3) at (2,-1) {3};
\node (4) at (0,-1) {4};

\draw (1) -- (2) -- (3) -- (4) -- (1);

\draw (3) -- (1);

\end{tikzpicture}}
\quad
\vec{A}_* =
\begin{bmatrix}
\frac{1}{7}&\frac{5}{14}&\frac{1}{7}&\frac{5}{14}\\[2pt]
\frac{5}{14}&\frac{2}{7}&\frac{5}{14}&0\\[2pt]
\frac{1}{7}&\frac{5}{14}&\frac{1}{7}&\frac{5}{14}\\[2pt]
\frac{5}{14}&0&\frac{5}{14}&\frac{2}{7}
\end{bmatrix},
$$
and $\frac{1}{2}\tr(\vec{A}_*\vec\Sigma \vec{A}_*^\top) = \frac{19}{21} \approx 0.905$. Adding positive correlation between the losses makes it harder to effectively share risk due to the increased chance of simultaneous loss, which results in an increase of the objective function.
\subsubsection{Negative  correlated losses} \label{sec:negcorlosses}
Finally, we modify the previous example by changing the covariance matrix $\vec\Sigma$ such that some losses are negatively correlated:
$$
\vec{\mu} = \vec{1} \quad 
\vec\Sigma =
\left[\begin{array}{rrrr}
    1&-\frac{1}{3}&0&-\frac{1}{3}\\[2pt]
    -\frac{1}{3}&1&-\frac{1}{3}&0\\[2pt]
    0&-\frac{1}{3}&1&-\frac{1}{3}\\[2pt]
    -\frac{1}{3}&0&-\frac{1}{3}&1
\end{array} \right]
 \quad
\raisebox{-.45\height}{\begin{tikzpicture}[scale=.75, every node/.style={circle, draw, fill=white, inner sep=2pt}]

\node (1) at (0,1) {1};
\node (2) at (2,1) {2};
\node (3) at (2,-1) {3};
\node (4) at (0,-1) {4};

\draw (1) -- (2) -- (3) -- (4) -- (1);

\draw (3) -- (1);

\end{tikzpicture}}
\quad
\vec{A}_* =
\begin{bmatrix}
\frac{5}{23}&\frac{13}{46}&\frac{5}{23}&\frac{13}{46}\\[2pt]
\frac{13}{46}&\frac{10}{23}&\frac{13}{46}&0\\[2pt]
\frac{5}{23}&\frac{13}{46}&\frac{5}{23}&\frac{13}{46}\\[2pt]
\frac{13}{46}&0&\frac{13}{46}&\frac{10}{23}
\end{bmatrix},
$$
and  $\frac{1}{2}\tr(\vec{A}_*\vec\Sigma \vec{A}_*^\top) = \frac{19}{69} \approx
0.275$. The negative correlations make risk-sharing more effective, so the objective function decreases relative to the previous example.

\begin{remark}
So far, the basic examples of $\vec{A}_*$ above have been symmetric matrices with nonnegative entries. We emphasize that neither property holds in general.
For additional, more general, examples of Theorem \ref{MainResults} see  \S \ref{sec:examples} and \S \ref{sec:barbell}.
\end{remark}

\subsection{Friends take equal shares of risk}

In this section, we study a special case of the network optimization problem \eqref{opteq}, where friends of an agent take on an equal share of the agent's risk. Due to this additional restriction, the objective function will, in general, increase compared to the previous unrestricted model, in exchange for equitably distributing risk between friends.

\begin{definition}[Friends take an equal share of risk] \label{assumption2}
Let $H(\vec{X}) = \vec{A} \vec{X}$ be a linear risk-sharing rule, for matrix $\vec{A} = (a_{ij}) \in \mathbb{R}^{n \times n}$. We say that the risk-sharing rule has the property that friends take an equal share of risk if
\begin{equation}
\label{equalshare}
a_{i_1j} = a_{i_2j}, \quad \forall \{i_1,j\},\{i_2,j\} \in E.
\end{equation}
\end{definition}

Informally speaking, this definition ensures that each of the friends of agent $j$ takes an equal share of the loss $X_j$ of agent $j$. In terms of the matrix $\vec{A}$, this definition imposes the condition that all the nonzero off-diagonal entries in each column are the same. This assumption is motivated by the network reciprocal contract with maximum contribution as studied in \cite{charpentier2021}. 
Taking these restrictions into account, the optimization problem for the risk-sharing rule $H(\vec{X}) = \vec{A} \vec{X}$ becomes
\begin{equation} \label{eq:opttakesameshare}
\begin{cases}
\text{minimize} & \frac{1}{2}\tr(\vec{A} \vec{\Sigma} \vec{A}^\top) \\
\text{subject to} & \vec{1}^\top \vec{A} = \vec{1}^\top, \quad  \vec{A} \vec{\mu}  = \vec{\mu} , \quad 
a_{ij} = 0 \text{ whenever } \{i,j\} \in \overline{E},
\\
& a_{i_1j} = a_{i_2j}, \quad \forall \{i_1,j\},\{i_2,j\} \in E,
\end{cases}
\end{equation}
where the minimization is taken over all $\vec{A} = (a_{i j}) \in \mathbb{R}^{n \times n}.$

Before stating our second main result, we review some standard concepts related to the {\em{graph Laplacian}} $\vec{L}$ of a graph.  For this, 
recall that the {\em{adjacency matrix}} $\vec W = (w_{i j}) \in \mathbb{R}^{n \times n}$ of the graph $G = (V,E)$ with vertices $V = \{1,\ldots,n\}$ is the $n\times n$ matrix with entries 
$$
w_{ij} = \begin{cases}
1, & \{i,j\}\in E, \\
0, & \text{otherwise}.
\end{cases}
$$ 
The {\em{graph Laplacian}} $\vec{L}$ is defined as
\begin{equation}
\label{Def:graphlapacian}
\vec{L= D - W},
\end{equation}
where $\vec D = \diag(d_1,\ldots,d_n)$ is the diagonal matrix with diagonal entries $d_i$ given by the degree of agent $i$. Recall that a graph is connected if there is a path between any two vertices, where a path between $i$ and $j$ is a sequence of edges $\{i_1,i_2\}$, $\{i_2,i_3\}$, $\{i_3,i_4\},
\ldots, \{i_{m-1},i_m\}$ such that $i=i_1$ and $j=i_m$. When the underlying graph is connected, the graph Laplacian $\vec L$ has a one-dimensional null space spanned by the all-ones vector $\vec 1$, see  \cite[Chapter 1.3]{chung1997spectral}. 

Our next result characterizes the solution to the optimization problem \eqref{eq:opttakesameshare} for the risk-sharing rule $H(\vec{X}) = \vec{A} \vec{X}$  for connected graphs.

\begin{theorem}[Friends take an equal share of risk]
\label{thmoptshare}
The optimization problem \eqref{eq:opttakesameshare} has a unique solution
\begin{equation} \label{hatAdef}
\vec{\hat{A}} = \vec{I} - \hat{c} \vec{L} \vec{M}^{-1},
\quad \text{for} \quad 
\hat{c} = \frac{\tr \left( \vec{\Sigma} \vec{L} \vec{M}^{-1} \right)}{\tr \left( \vec{L} \vec{M}^{-1} \vec{\Sigma} \vec{M}^{-1} \vec{L} \right)},
\end{equation}
where $\vec{M} = \diag(\mu_1,\ldots,\mu_n)$, and $\vec{L}$ is the graph Laplacian.
\end{theorem}

Note that the matrix $\vec{M}^{-1}$ is well defined since $\vec{\mu}$ has positive entries; see
Section \ref{sec:prelmin}. Moreover, the constant $\hat{c}$ is well defined since the term in the denominator of the fraction defining it is positive. Indeed, we have
\begin{equation} \label{eq:justifytracepos}
\tr \left( \vec{L} \vec{M}^{-1} \vec{\Sigma} \vec{M}^{-1} \vec{L} \right) = \tr  
\left( (\vec{\Sigma}^{1/2} \vec{M}^{-1} \vec{L})^\top
(\vec{\Sigma}^{1/2} \vec{M}^{-1} \vec{L}) \right)
= \| \vec{\Sigma}^{1/2} \vec{M}^{-1} \vec{L} \|_F^2,
\end{equation}
where the final equality follows from the trace formula for the Frobenius norm. Since $\vec{\Sigma}^{1/2} \vec{M}^{-1}$ is of full rank and 
$\vec{L}$ has rank $n-1$, it follows
that the product 
$\vec{\Sigma}^{1/2} \vec{M}^{-1} \vec{L}$ has rank $n-1$.
Since we assume $n \ge 2$ throughout the paper,  its Frobenius norm
$\|\vec{\Sigma}^{1/2} \vec{M}^{-1} \vec{L}\|_F^2 > 0$.

\begin{proof}[Proof of Theorem \ref{thmoptshare}]
First, we derive the general form of a risk-sharing matrix $\vec{A} = (a_{ij}) \in \mathbb{R}^{n \times n}$ that satisfies the constraints of \eqref{eq:opttakesameshare}.
Together, the conditions
$a_{ij} = 0$ whenever $\{i,j\} \in \overline{E}$
and
$a_{i_1j} = a_{i_2j}, \forall \{i_1,j\},\{i_2,j\} \in E$
imply that 
$$
a_{ij} = \begin{cases}
t_j & \text{for } i=j \\
s_j & \text{for }\{i,j\} \in E \\
0 & \text{otherwise},
\end{cases}
$$
for some values $t_1,\ldots,t_n \in \mathbb{R}$ and $s_1,\ldots,s_n \in \mathbb{R}$.
Further taking into account the constraint $\vec{1}^\top \vec{A} = \vec{1}^\top$, we have
$$
t_i = 1 - d_i s_i,
$$
where $d_i$ is the degree of vertex $i$. So, in matrix notation, $\vec{A}$ is of the form
\begin{equation} \label{eq:ILS}
\vec{A} = \vec I - \vec{L} \vec{S},
\end{equation}
where $\vec{L}$ is the graph Laplacian defined in \eqref{Def:graphlapacian}, and $\vec{S} = \diag(s_1,\ldots,s_n)$.
Next, we consider the constraint
$\vec{A} \vec{\mu}  = \vec{\mu}$, which together with \eqref{eq:ILS} yields the equation
$$ 
\vec{\mu} - \vec{L} \vec{S} \vec{\mu} = \vec{\mu},
$$
or equivalently $\vec{L} \vec{S} \vec{\mu} = \vec{0}$. 
As the underlying graph is connected by assumption, the Laplacian $\vec L$ has a one-dimensional null space spanned by the all-ones vector $\vec 1$,  as noted above.
Thus, the risk-sharing rule is actuarially fair only when $\vec{S} \vec{\mu} = c \vec{1}$ for some constant $c \in \mathbb{R}$. It follows that 
\begin{equation} \label{formofsoln}
\vec{A} = \vec{I} - c \vec{L} \vec{M}^{-1},
\end{equation}
where $\vec{M} = \diag(\mu_1,\ldots,\mu_n)$ and $c \in \mathbb{R}$. 

Next, we optimize the parameter $c \in \mathbb{R}$.
Substituting \eqref{formofsoln} into the objective function 
of the optimization problem \eqref{eq:opttakesameshare}, and
using the linearity of the trace, and the fact that the trace of a matrix is equal to the trace of its transpose, gives
$$
\frac{1}{2}\tr(\vec{A} \vec{\Sigma} \vec{A}^\top)  = \frac{1}{2} \tr \left( \vec{\Sigma} \right) - c \tr( \vec{L} \vec{M}^{-1} \vec{\Sigma})  + \frac{1}{2} c^2 \tr ( \vec{L} \vec{M}^{-1} \vec{\Sigma} \vec{M}^{-1} \vec{L} ).
$$
Recall that we established $\tr \left( \vec{L} \vec{M}^{-1} \vec{\Sigma} \vec{M}^{-1} \vec{L} \right) > 0$ in
\eqref{eq:justifytracepos}, so the right-hand side is a strictly convex function of the variable $c$.
Setting the derivative of this expression with respect to $c$ equal to zero and solving for the critical point $\hat{c}$ gives
\begin{equation} \label{eq:critical}
\hat{c} = \frac{\tr \left( \vec{\Sigma} \vec{L} \vec{M}^{-1} \right)}{\tr \left( \vec{L} \vec{M}^{-1} \vec{\Sigma} \vec{M}^{-1} \vec{L} \right)}.
\end{equation}
Since the objective function is strictly convex, this critical point is its minimum, and the proof is complete.
\end{proof}

Note that if the graph $G$ is not connected, then this theorem can be applied to each connected component of the graph.

\begin{remark}[Special cases of Theorem \ref{thmoptshare}]\label{optshareuncorrelated}
When the losses are uncorrelated and $\vec{\Sigma} = \diag(\sigma_1^2,\ldots,\sigma_n^2)$ we have
$$
\hat{c} = 
\frac{\sum_{i=1}^n d_i \mu_i^{-1} \sigma_i^2}{ \sum_{i=1}^n (d_i^2 + d_i) \sigma_i^2 \mu_i^{-2}}.
$$
Furthermore, if the underlying graph is $d$-regular (meaning that each vertex $i$ has degree $d_i=d$), and the loss random vector $\vec{X}$ has i.i.d. entries such that the mean $\vec\mu=\mu\vec1$ and variance $\vec\Sigma=\sigma^2\vec{I}$, then 
$$
\hat{c} =  \frac{ d \mu^{-1} \sigma^2}{ (d^2 + d) \sigma^2 \mu^{-2}} = \frac{\mu}{d+1},
\qquad \text{and} \qquad
\vec{\hat{A}} = \vec{I} - \frac{1}{d+1} \vec{L}, 
$$
which corresponds to each agent $i$ sharing their loss equally with their friends: agent $i$ keeps $1/(d+1)$ of their own loss, and each of their $d$ friends takes $1/(d+1)$ of agent $i$'s loss. 
\end{remark}

\subsection{Examples of Theorem \ref{thmoptshare}} \label{sec:examplesAhat}
In the following, we revisit three examples from 
\S \ref{sec:examples1} to illustrate how $\vec{\hat{A}}$ from Theorem \ref{thmoptshare} differs from $\vec{A}_*$ from Theorem \ref{MainResults}.
In each example, we restate the mean vector $\vec{\mu}$, covariance matrix $\vec{\Sigma}$, and graph $G= (V,E)$, and use Theorem \ref{thmoptshare}
to determine the unique $\vec{\hat{A}}$ that solves the optimization problem \eqref{eq:opttakesameshare}, and report the value of $\frac{1}{2}\tr(\vec{\hat{A}} \vec{\Sigma} \vec{\hat{A}}^\top )$.

\subsubsection{Complete graph}
For the complete graph with uncorrelated losses with unit mean and variance, see 
\S \ref{sec:completegraph}, the optimal solution $\vec{A}_*$ already fulfills the condition that friends take an equal share of risk, so the optimal risk-sharing matrix $\vec{\hat{A}}$ does not change from $\vec{A}_*.$ This example shows how networks in the $d$-regular and i.i.d. case of Remark \ref{optshareuncorrelated} already share risk equally among friends.

\subsubsection{Complete graph with one edge removed}
Next, we revisit the example from
\S \ref{sec:oneedgeremoved} of a complete graph with one edge removed and uncorrelated losses. Here, we have
$$
\vec{\mu} = \vec{1} \quad \vec{\Sigma} = \vec{I} \quad
\raisebox{-.45\height}{\begin{tikzpicture}[scale=.75, every node/.style={circle, draw, fill=white, inner sep=2pt}]

\node (1) at (0,1) {1};
\node (2) at (2,1) {2};
\node (3) at (2,-1) {3};
\node (4) at (0,-1) {4};

\draw (1) -- (2) -- (3) -- (4) -- (1);

\draw (3) -- (1);

\end{tikzpicture}}
\quad
\vec{\hat{A}} =\begin{bmatrix}
    \frac{1}{6}&\frac{5}{18}&\frac{5}{18}&\frac{5}{18}\\[2pt]
    \frac{5}{18}&\frac{4}{9}&\frac{5}{18}&0\\[2pt]
    \frac{5}{18}&\frac{5}{18}&\frac{1}{6}&\frac{5}{18}\\[2pt]
    \frac{5}{18}&0&\frac{5}{18}&\frac{4}{9}
\end{bmatrix},
$$
and $\frac{1}{2}\tr(\vec{\hat{A}}\vec\Sigma\vec{\hat{A}}^\top) = \frac{11}{18}
\approx 0.611 $.
The optimal trace slightly increases compared to the original example, while the new optimized risk-sharing matrix now meets the friends take equal shares of risk condition.
\subsubsection{Positive correlation} \label{sec:poscor2}
Next, we revisit the case of the network with positively correlated losses and one missing edge (\S \ref{sec:poscorr}).
Here, we have
$$
\vec{\mu} = \vec{1} \quad 
\vec\Sigma =
\begin{bmatrix}
    1&\frac{1}{3}&0&\frac{1}{3}\\[2pt]
    \frac{1}{3}&1&\frac{1}{3}&0\\[2pt]
    0&\frac{1}{3}&1&\frac{1}{3}\\[2pt]
    \frac{1}{3}&0&\frac{1}{3}&1
\end{bmatrix}
 \quad
\raisebox{-.45\height}{\begin{tikzpicture}[scale=.75, every node/.style={circle, draw, fill=white, inner sep=2pt}]

\node (1) at (0,1) {1};
\node (2) at (2,1) {2};
\node (3) at (2,-1) {3};
\node (4) at (0,-1) {4};

\draw (1) -- (2) -- (3) -- (4) -- (1);

\draw (3) -- (1);

\end{tikzpicture}}
\quad
\vec{\hat{A}} =
\begin{bmatrix}
\frac{5}{38}&\frac{11}{38}&\frac{11}{38}&\frac{11}{38}\\[2pt]
\frac{11}{38}&\frac{8}{19}&\frac{11}{38}&0\\[2pt]
\frac{11}{38}&\frac{11}{38}&\frac{5}{38}&\frac{11}{38}\\[2pt]
\frac{11}{38}&0&\frac{11}{38}&\frac{8}{19}
\end{bmatrix},
$$
and $\frac{1}{2}\tr(\vec{\hat{A}}\vec\Sigma\vec{\hat{A}}^\top) = \frac{107}{114} \approx 0.938.$
Once again, there is a slight increase in the objective, relative to the result in \S \ref{sec:poscorr}, as a consequence of the additional condition that friends take an equal share of risk. 

\subsubsection{Negative correlation} \label{sec:negcor}
Finally, we revisit the case of the network with negatively correlated losses and one missing edge 
(\S \ref{sec:negcorlosses}). Here, we have
$$
\vec{\mu} = \vec{1} \quad 
\vec\Sigma =
\begin{bmatrix}
    1&-\frac{1}{3}&0&-\frac{1}{3}\\[2pt]
    -\frac{1}{3}&1&-\frac{1}{3}&0\\[2pt]
    0&-\frac{1}{3}&1&-\frac{1}{3}\\[2pt]
    -\frac{1}{3}&0&-\frac{1}{3}&1
\end{bmatrix}
 \quad
\raisebox{-.45\height}{\begin{tikzpicture}[scale=.75, every node/.style={circle, draw, fill=white, inner sep=2pt}]

\node (1) at (0,1) {1};
\node (2) at (2,1) {2};
\node (3) at (2,-1) {3};
\node (4) at (0,-1) {4};

\draw (1) -- (2) -- (3) -- (4) -- (1);

\draw (3) -- (1);

\end{tikzpicture}}
\quad
\vec{\hat{A}} = 
\begin{bmatrix}
    \frac{13}{70}&\frac{19}{70}&\frac{19}{70}&\frac{19}{70}\\[2pt]
    \frac{19}{70}&\frac{16}{35}&\frac{19}{70}&0\\[2pt]
    \frac{19}{70}&\frac{19}{70}&\frac{13}{70}&\frac{19}{70}\\[2pt]
    \frac{19}{70}&0&\frac{19}{70}&\frac{16}{35}
\\\end{bmatrix},
$$
and $\frac{1}{2}\tr(\vec{\hat{A}}\vec\Sigma\vec{\hat{A}}^\top) \approx 0.281$. Once again, there is a slight increase in the objective, relative to the result in \S \ref{sec:negcorlosses} as a consequence of the additional condition that friends take an equal share of risk.

\subsection{Nonnegativity Conditions} \label{sec:nonneg}

In this section, we develop conditions for the nonnegativity of the optimal risk-sharing matrices $\vec{A}_*$ and $\vec{\hat{A}}$. While these matrices sometimes naturally have nonnegative entries (such as in the examples considered in \S \ref{sec:examples1} and \S \ref{sec:examplesAhat}), in general, negative entries may arise in both optimization problems. 
As noted in Section 1.2, a risk-sharing scheme involving signed exchanges is analogous to financial models that allow agents to buy and short-sell financial instruments.  Allowing signed transfers is mathematically convenient, since they allow the optimization problems to be formulated as equality-constrained quadratic programs. By not imposing nonnegativity constraints on the risk-sharing matrix in the previous sections, we obtain the explicit solutions in Theorems \ref{MainResults} and \ref{thmoptshare}. 

If one additionally requires $a_{ij}\geq 0$, then the optimization problem becomes a quadratic program with both equality and inequality constraints. In this case, the optimal solution is characterized by the KKT conditions together with complementary slackness, and 
the solution can be determined using a quadratic program solver, but explicit formulas of the same type are generally unavailable. It is therefore useful to determine necessary and sufficient conditions that guarantee, a priori, the non-negativity of the entries of the optimal risk-sharing matrices $\vec{A}_*$ and $\vec{\hat{A}}$.   

Furthermore, as illustrated below, these conditions involve naturally related statistical quantities associated with the risk characteristics of individual participants 
and with global statistical quantities.
In the statements below, the non-dimensional quantity 
$$\rho^2 = \vec\mu^T \vec\Sigma^{-1} \vec\mu$$ 
naturally arises.  Its scalar version is the square of the Return to Risk, $\left(\tfrac{\mu}{\sigma}\right)^2,$ of a random variable $X$ with mean $\mu$ and standard deviation $\sigma$.  Thus $\rho^2$ is the corresponding quantity for a random vector $\vec X$ with mean $\vec\mu$ and  a positive definite covariance matrix $\vec\Sigma$.  We will refer to $\rho$ as the collective Return to Risk ratio.

\subsubsection{ Nonnegativity Conditions for $\vec{A}_*$ for the case of the complete graph}
\label{sec:nonnegAstar}

In the following lemmas, we state conditions on $\vec{\mu}$ and $\vec{\Sigma}$ 
such that the optimal $\vec{A}_*$ defined by Theorem \ref{thm:feng} (or equivalently, defined by Theorem \ref{MainResults} for the case of the complete graph) has nonnegative entries.  We start by considering the simplified case where $\vec{\Sigma}$ is a scaled identity matrix, before considering the general case.

\begin{lemma}\label{MuConditions}
Assume that $\vec\Sigma = \sigma^2 \vec I$ for some $\sigma\neq 0,$ and let $\underline{\vec\mu} := (\min_{1\le j \le n} \mu_j) \vec{1}.$   Then, all entries of the matrix $\vec{A}_*$  defined in Theorem \ref{thm:feng} are nonnegative if and only if
$$
\|\vec\mu-\underline{\vec\mu}\|_1 \|\vec\mu\|_\infty \le \|\vec\mu\|_2^2.
$$
\end{lemma}

\begin{proof}[Proof of Lemma \ref{MuConditions}]
Since $\vec{\Sigma} = \sigma^{{2}} \vec{I}$, we have
$$
\vec{A}_* = \frac{1}{n} \vec{1} \vec{1}^\top + \left( \vec{I} - \frac{1}{n} \vec{1} \vec{1}^\top \right) \frac{\vec{\mu} \vec{\mu}^\top}{\vec{\mu}^\top \vec{\mu}}.
$$
Thus, all of the entries of $\vec{A}_*$ are nonnegative if and only if
$$
0\le \frac{1}{n} +\frac{\mu_j}{\|\vec\mu\|_2^2} \left(\mu_i - \frac{1}{n} \|\vec\mu\|_1\right), \quad \forall i,j \in \{1,\ldots,n\}.
$$
Rearranging terms, the condition for nonnegative entries is 
 \begin{equation} \label{eq:conditionfornonneg}
\mu_j (\|\vec\mu\|_1 - n \mu_i) \le  \|\vec\mu\|_2^2, \quad \forall i,j \in \{1,\ldots,n\}.
\end{equation}
The result follows from maximizing the left-hand side over $i$ and $j$. Indeed, the maximum over $i$ occurs when 
$$
\|\vec{\mu}\|_1 - n \mu_i = \|\vec\mu\|_1 - n \left(\min_{1\le k \le n} \mu_k \right) =
\|\vec\mu-\underline{\vec\mu}\|_1,
$$
and the maximum over $j$ occurs when $\mu_j=\|\vec\mu\|_\infty$.
\end{proof}

In the case $\vec\Sigma = \sigma^2 \vec I$,  $\rho^2= \tfrac{1}{\sigma^2}  \|\vec\mu\|_2^2.$ The nonnegativity conditions stated in Lemma \ref{MuConditions} can be written as
$$
\frac{\| \vec\mu -\underline{\vec\mu}\|_1 \|\vec\mu\|_\infty}{\sigma^2} \le \rho^2,
$$ 
where $\rho$ is the collective Return to Risk ratio.
One can interpret the left hand side as a measure of the variability of the expected losses normalized by their variance.  Note that the dimensionality of the loss vector  $\vec{X}$ is reflected in the use of the vector norms $\|\vec\mu - \underline{\vec\mu}\|_1$ and $\|\vec\mu\|_2.$

Next, we consider the case where $\vec{\Sigma}$ is an arbitrary positive definite matrix.

\begin{proposition}
\label{lem:gencase}
Assume that $\vec{\Sigma}$ is positive definite. Let 
$\underline{\vec\mu} := (\min_{1\le j \le n} \mu_j) \vec{1}$ and $\overline{\vec\mu} := (\max_{1\le j \le n} \mu_j) \vec{1}$. Define
$$
c_0 :=\min_{1\le j \le n}(\vec\Sigma^{-1}\vec\mu)_j \quad \text{and}  \quad c_1:=\max_{1\le j \le n}(\vec\Sigma^{-1}\vec\mu)_j. 
$$ 
Then, all the entries of the matrix $\vec{A}_*$  defined in Theorem \ref{thm:feng} are nonnegative if and only if
$$
\max \left\{ -c_0 \|\vec{\mu} - \overline{\vec{\mu}}\|_1, c_1 \|\vec{\mu} - \underline{\vec{\mu}}\|_1 \right\} \le \vec{\mu}^\top \vec{\Sigma}^{-1} \vec{\mu}.
$$
\end{proposition}
\begin{proof}[Proof of Proposition \ref{lem:gencase}]
Recall that 
$$
\vec{A}_* =\frac{1}{n} \vec{1} \vec{1}^\top +\frac{1}{\vec{\mu}^\top \vec{\Sigma}^{-1} \vec{\mu}}\left(\vec I- \frac{1}{n} \vec{1} \vec{1}^\top\right)\vec{\mu \mu}^\top \vec\Sigma^{-1}.
$$
Thus, all of the entries of $\vec{A}_*$ are nonnegative if and only if
$$
0 \le \frac{1}{n} +  \frac{\left(  \vec{\Sigma}^{-1} \vec{\mu} \right)_j}{\vec{\mu}^\top \vec{\Sigma}^{-1} \vec{\mu}}  \left( \mu_i - \frac{1}{n}\|\vec{\mu}\|_1 \right) \quad \forall i,j \in \{1,\ldots,n\}.
$$
Rearranging terms, the condition for nonnegative entries is
$$
\left(  \vec{\Sigma}^{-1} \vec{\mu} \right)_j \left(  \|\vec{\mu}\|_1 - n \mu_i \right)  \le \vec{\mu}^\top \vec{\Sigma}^{-1} \vec{\mu}
\quad \forall i,j \in \{1,\ldots,n\}.
$$
We proceed as before, maximizing over $i, j$ and denote by $i^*, j^*$ the indices at which the maximum occurs.  If $\left(\vec\Sigma^{-1}\vec \mu\right)_{j^{{*}}} \geq 0,$ then
$$
\max_{1 \le i,j \le n} \left( \left(  \vec{\Sigma}^{-1} \vec{\mu} \right)_j \left( \|\vec{\mu}\|_1 - n \mu_i \right) \right)= (\vec\Sigma^{-1}\vec\mu)_{j^*}
\|\vec{\mu} - \underline{\vec{\mu}}\|_1.
$$
Otherwise, if  $\left(  \vec{\Sigma}^{-1} \vec{\mu} \right)_{j^*} < 0$, then
$$
\max_{1 \le i,j \le n} \left( \left(  \vec{\Sigma}^{-1} \vec{\mu} \right)_j \left( \|\vec{\mu}\|_1 - n \mu_i \right)\right) = -(\vec\Sigma^{-1}\vec\mu)_{j^*} 
\|\vec{\mu} - \overline{\vec{\mu}}\|_1.
$$
Taking the maximum of these two cases gives the result.
\end{proof}

The conditions for nonnegativity of the entries of $\vec{A}_*$ have a similar interpretation as above.  However, the covariance matrix is involved in the determination of $c_0, c_1$ representing, respectively, the smallest and largest entry of the (dimensional) vector $\vec\nu=\vec{\Sigma}^{-1}\vec\mu$.  We remark that $\vec\nu$ also arises in the resolution of the Markowitz quadratic optimization problem in the determination of the `market' weights of the optimal portfolio, see e.g. \cite{markowitz}.  In the context of this paper, $\vec\nu$ appears as a measure of the variability of the loss vector $\vec X$ with strictly positive mean $\vec\mu$ and covariance $\vec\Sigma$, relative to the square of the collective Return to Risk ratio $\rho^2.$

\subsubsection{Nonnegativity conditions for $\vec{\hat{A}}$} \label{sec:nonnegAhat}
In the following, we consider a condition for the nonnegativity of the optimal risk-sharing matrix $\vec{\hat{A}}$ defined in Theorem \ref{thmoptshare} for the case where friends take an equal share of risk.

\begin{lemma} \label{FriendsEqualPos}
Under the hypotheses of Theorem \ref{thmoptshare}, assume that $\vec{\hat{A}} = \vec{I} - \hat{c} \vec{L} \vec{M}^{-1}$ is defined
by \eqref{hatAdef}. Then $\vec{\hat{A}}$ has all nonnegative entries if and only if
$$
1 \ge  \hat{c}\frac{d_i}{\mu_i} \ge 0, \quad 
$$ 
for all $1\leq i\leq n.$

\end{lemma}
\begin{proof}
Let $\vec{\hat{A}} = \vec{I}-\hat{c}\vec{LM}^{-1}$, where $\hat{c}$ is defined by \eqref{hatAdef}.
First, we consider the off-diagonal entries of $\vec{\hat{A}}$. When $i \not = j$ we have
$$
(\vec{\hat{A}})_{ij} = (\vec{I} - \hat{c} \vec{LM}^{-1})_{ij} = \begin{cases}
        \frac{\hat{c}}{\mu_j},& \{i,j\}\in E,
        \\0,& \{i,j\}\notin E.
    \end{cases}
$$
Since we assumed the underlying graph is connected and consists of $n \ge 2$ nodes, there is at least one edge, and since the losses are nonnegative random variables with positive variance, each mean $\mu_j > 0$. It follows that the off-diagonal entries of $\vec{\hat{A}}$ are nonnegative if and only if $\hat{c} \ge 0$.
Next, we consider the diagonal entries of $\vec{\hat{A}}$. We have 
$$
(\vec{\hat{A}})_{ii} = (\vec{I} - \hat{c} \vec{LM}^{-1})_{ii} = 1 - \hat{c} \frac{d_i}{\mu_i},
$$
which is nonnegative if and only if $\hat{c} (d_i/\mu_i) \le 1$. Combining the inequalities from the off-diagonal and diagonal cases completes the proof.
\end{proof}

Let $\Cov(X,Y) = \mathbb{E}\left[ \left(X - \mathbb{E}[X] \right) \left( Y - \mathbb{E}[Y] \right) \right]$ denote the covariance of $X$ and $Y$.

\begin{corollary}\label{cstarcovarianceeq}
The condition $\hat{c} \geq 0$ is equivalent to the condition
$$
 \sum_{\{i,j\}\in E}\Cov(X_i,X_j) \left( \frac{1}{\mu_i}+\frac{1}{\mu_j} \right) \leq \sum_{i=1}^n \frac{d_i\sigma_i^2}{\mu_i}.
$$
\end{corollary}
\begin{proof}
Recall that
$$
\hat{c} = \frac{\tr \left( \vec{\Sigma} \vec{L} \vec{M}^{-1} \right)}{\tr \left( \vec{L} \vec{M}^{-1} \vec{\Sigma} \vec{M}^{-1} \vec{L} \right)}.
$$
The term in the denominator of this expression
is strictly positive by 
\eqref{eq:justifytracepos}. In the numerator, we have $$\tr(\vec{\Sigma L M}^{-1}) = \sum_{i=1}^n \frac{d_i\sigma^2_i-\sum_{j : \{i,j\} \in E}\Cov(X_i,X_j)}{\mu_i}.$$ 
For this sum to be nonnegative, we require 
$$
\sum_{\{i,j\}\in E}\Cov(X_i,X_j)\left(\frac{1}{\mu_i}+\frac{1}{\mu_j} \right) \leq \sum_{i=1}^n \frac{d_i\sigma_i^2}{\mu_i},
$$
as was to be shown.
\end{proof}

\begin{remark}
As a consequence of Corollary \ref{cstarcovarianceeq}, we note that a sufficient condition for the nonnegativity of $\hat{c}$ is that 
for all $\{i,j\}\in E,$ $\Cov(X_i,X_j)\leq \min\{\sigma_i^2,\sigma_j^2\}.$ 
\end{remark}

The above corollary and remark demonstrate how the nonnegativity of off-diagonal entries is dependent on covariance relative to expected losses.
More specifically, the quantity 
${\tr \left( \vec{L} \vec{M}^{-1} \vec{\Sigma} \vec{M}^{-1} \vec{L} \right)}$,
shown to be positive in 
\eqref{eq:justifytracepos},
that appears in the denominator of $\hat{c}$, can be interpreted as the square of the {\it{Network Coefficient of Variation}}
(NCV) as it generalizes the corresponding standard notion for scalar random variables.  One can interpret the conditions for nonnegativity in Lemma \ref{FriendsEqualPos} as stating that for each agent, their own {\it{network adjusted coefficient of variation}} defined as 
$d_i \tr\left(\vec\Sigma\vec L \vec M^{-1}\right) \mu_i^{-1}$ 
needs to be nonnegative and at most the square of the NCV.

In particular, we note the following result for a simple 2-agent network as it will serve as the basis for some examples in \S \ref{sec:examples}.  Here, a 2-agent network has graph $G = (V,E),$ with vertices $V = \{1,2\}$ and edges $E = \{ \{1,2\} \}.$

\begin{corollary} \label{cor:cor2agent}
    In the case of a 2-agent network, 
   $\vec{\hat{A}}$ has all nonnegative entries if and only if $\hat{c}\leq \mu_i$ for $i \in\{1,2\}$ and $$\Cov(X_1,X_2) \leq \frac{\sigma_1^2\mu_2+\sigma_2^2\mu_1}{\mu_1+\mu_2}.$$
\end{corollary}
Corollary \ref{cor:cor2agent} is an immediate consequence of Lemma \ref{FriendsEqualPos} and Corollary \ref{cstarcovarianceeq}.

\subsection{Examples where optimal matrices have negative entries} \label{sec:examples}
In this section, we provide some examples where optimal risk-sharing matrices $\vec{A}_*$ and $\vec{\hat{A}}$ have entries that are negative.  Additionally, we illustrate how network-based restrictions can eliminate negative entries in some cases.

\subsubsection{ Agents with losses with means at different scales} \label{examplenegative}
First, we consider an example where a negative entry appears due to the conditions of Lemma \ref{MuConditions} being violated. In particular, we consider a complete graph on three vertices with losses whose covariance matrix is the identity, but that have means at different scales. We have
$$
\vec{\mu} = \begin{bmatrix}
\frac{1}{4} \\
1 \\ 
4
\end{bmatrix} 
\quad \vec{\Sigma} = \vec{I} \quad
\raisebox{-.45\height}{\begin{tikzpicture}[scale=.75, every node/.style={circle, draw, fill=white, inner sep=2pt}]

\node (1) at (1,1) {1};
\node (2) at (2,-1) {2};
\node (3) at (0,-1) {3};

\draw (1) -- (2) -- (3) -- (1);

\end{tikzpicture}}
\quad
\vec{A}_* = \left[\begin{array}{rrr}
\frac{85}{273} & \frac{67}{273} & -\frac{5}{273} \\[2pt] \frac{88}{273} & \frac{79}{273} & \frac{43}{273} \\[2pt] \frac{100}{273} & \frac{127}{273} & \frac{235}{273}
\end{array} \right],
$$
and $\frac{1}{2}\tr(\vec{A}_*\vec\Sigma \vec{A}_*^\top) = \frac{19}{26} \approx0.731$. Recall that, when the covariance matrix $\vec\Sigma$ is a multiple of the identity, we can apply Lemma \ref{MuConditions}, and thus the entries of $\vec{A}_*$ are nonnegative if and only if 
$$
\|\vec\mu-\underline{\vec\mu}\|_1 \|\vec\mu\|_\infty \le \|\vec\mu\|_2^2
\quad \text{where} \quad \underline{\vec\mu} := \left(\min_{1\le j \le n} \mu_j\right) \vec{1}.
$$
Here, 
$$
\|\vec\mu-\underline{\vec\mu}\|_1 \|\vec\mu\|_\infty =
\left\|\vec\mu-\frac{1}{4}\vec1\right\|_1 (4) = 18 >  17.0625
 = \frac{273}{16}=\|\vec\mu\|_2^2,
$$
so the necessary and sufficient condition of Lemma \ref{MuConditions} is not met.

\subsubsection{Restricting risk-sharing to maintain nonnegativity}
\label{sec:disconnnect}
By adding the restriction that agents $1$ and $3$ cannot exchange risk, we
can eliminate the negative entry from the previous example. We have
$$
\vec{\mu} = \begin{bmatrix}
\frac{1}{4} \\
1 \\ 
4
\end{bmatrix} 
\quad \vec{\Sigma} = \vec{I} \quad
\raisebox{-.45\height}{\begin{tikzpicture}[scale=.75, every node/.style={circle, draw, fill=white, inner sep=2pt}]

\node (1) at (1,1) {1};
\node (2) at (2,-1) {2};
\node (3) at (0,-1) {3};

\draw (1) -- (2) -- (3);
\end{tikzpicture}}
\quad
\vec{A}_*= \begin{bmatrix}\frac{18}{38}&\frac{5}{38}&0\\[2pt]
\frac{20}{38}&\frac{13}{38}&\frac{5}{38}\\[2pt]
0&\frac{20}{38}&\frac{33}{38}\end{bmatrix},
$$
and $\frac{1}{2}\tr(\vec{A}_*\vec\Sigma \vec{A}_*^\top) = \frac{16}{19}
\approx 0.842$. Note that since we have restricted the choices of $\vec{A}_*$, the objective function
$\frac{1}{2}\tr(\vec{A}_*\vec\Sigma \vec{A}_*^\top)$
has slightly increased, but now we maintain nonnegativity.

\subsubsection{Sharing risk equally among friends to maintain nonnegativity}
Here, we show a different way to achieve nonnegativity 
that does not require altering the network structure as in \S \ref{sec:disconnnect}. Instead, we require friends take an equal share of risk (as in Definition \ref{assumption2}).
Here, $\vec\mu,$ $\vec\Sigma,$ and the network structure are the same as in \S \ref{examplenegative}, and we use Theorem \ref{thmoptshare} to compute the optimal risk-sharing matrix where friends take an equal share of risk. For this example, we have 
$$
\vec{\mu} = \begin{bmatrix}
\frac{1}{4} \\
1 \\ 
4
\end{bmatrix} 
\quad \vec{\Sigma} = \vec{I} \quad
\raisebox{-.45\height}{\begin{tikzpicture}[scale=.75, every node/.style={circle, draw, fill=white, inner sep=2pt}]

\node (1) at (1,1) {1};
\node (2) at (2,-1) {2};
\node (3) at (0,-1) {3};

\draw (1) -- (2) -- (3) -- (1);

\end{tikzpicture}}
\vec{\hat{A}} = \begin{bmatrix}
    \frac{7}{39}&\frac{4}{39}&\frac{1}{39}\\[2pt]
    \frac{16}{39}&\frac{31}{39}&\frac{1}{39}\\[2pt]
    \frac{16}{39}&\frac{4}{39}&\frac{37}{39}
\end{bmatrix}\quad
\hat{c}=\frac{4}{39},
$$
and  $\frac{1}{2}\tr(\vec{\hat{A}}\vec\Sigma\vec{\hat{A}}^\top) = \frac{25}{26} \approx 0.962.$
Note that although this network fails to meet the nonnegativity conditions as stated in Lemma \ref{MuConditions}, it satisfies those in Lemma \ref{FriendsEqualPos}.

While the condition that friends take an equal share of risk prevents a negative entry in this example, next we demonstrate two cases where friends take an equal share of risk, but the conditions of Lemma \ref{FriendsEqualPos} fail, which  results in negative entries in $\vec{\hat{A}}.$

\subsubsection{Negative off-diagonal entries}
Here we show an example where $\hat{c}$ is negative, and the condition of Corollary \ref{cor:cor2agent} is violated, which results in a negative entry for $\vec{\hat{A}}$. Here, we have
$$
\vec\mu=\begin{bmatrix}
    1\\5
\end{bmatrix} \quad 
\vec\Sigma = \begin{bmatrix}
    1&3\\[2pt]
    3&\frac{19}{2}
\end{bmatrix} \quad
\hat{c} = -\frac{35}{18} \quad
\vec{\hat{A}} = \begin{bmatrix}
    \frac{53}{18}&-\frac{7}{18}\\[2pt]
    -\frac{35}{18}&\frac{25}{18}
\end{bmatrix},
$$
and $\frac{1}{2}\tr(\vec{\hat{A}}\vec\Sigma\vec{\hat{A}}^\top) = \frac{329}{72}\approx 4.57.$
Indeed, referring to Corollary \ref{cor:cor2agent},
$$\Cov(X_1,X_2) = 3 > \frac{29}{12} = \frac{\sigma_1^2\mu_2+\sigma_2^2\mu_1}{\mu_1+\mu_2}.$$
Since the covariance value exceeds the bounds for a nonnegative $\hat{c},$ the off-diagonal entries are negative (see the Proof of Lemma \ref{FriendsEqualPos}).

\subsubsection{Negative diagonal entry}
Here we show an example where $\hat{c} > \frac{\mu_1}{d_1},$ resulting in a negative diagonal entry.
We have:
$$
\vec\mu=\vec1\quad \vec\Sigma = \diag(1,8,8,8) \quad
\raisebox{-.45\height}{\begin{tikzpicture}[scale=.75, every node/.style={circle, draw, fill=white, inner sep=2pt}]

\node (1) at (0,1) {1};
\node (2) at (2,1) {2};
\node (3) at (2,-1) {3};
\node (4) at (0,-1) {4};

\draw (1) -- (2);
\draw (3) -- (1);
\draw (1) -- (4);

\end{tikzpicture}}
\hat{c} = \frac{9}{20}\quad
\vec{\hat{A}} = \begin{bmatrix}
    -\frac{7}{20}&\frac{9}{20}&\frac{9}{20}&\frac{9}{20}\\[2pt]
    \frac{9}{20}&\frac{11}{20}&0&0\\[2pt]
    \frac{9}{20}&0&\frac{11}{20}&0\\[2pt]
    \frac{9}{20}&0&0&\frac{11}{20}
\end{bmatrix},
$$
and  $\frac{1}{2}\tr(\vec{\hat{A}}\vec\Sigma\vec{\hat{A}}^\top) =\frac{257}{40} = 6.425$. 
Indeed, when $i=1$ 
$$
\frac{\mu_i}{d_i} = \frac{1}{3} < \hat{c} = 
\frac{9}{20}
$$
so $\vec{\hat{A}}$ has a negative diagonal entry.

\subsection{Barbell network example}
\label{sec:barbell}
Here, we illustrate how organizing agents in a network structure where agents are connected if they have similar expected losses can avoid negative entries in the optimal risk-sharing matrix. Suppose that there are $6$ agents whose random losses have mean vector and covariance matrix
$$
\vec\mu = \begin{bmatrix}
    1&1&4&16&64&64
\end{bmatrix}^\top\quad
\text{and} \quad
\vec\Sigma = \vec I,
$$
respectively. First, we assume all agents are allowed to share risk in a fully-connected network, and we use Theorem \ref{thm:feng} to compute the optimal risk-sharing matrix $\vec{A}_*$, see Figure \ref{fig:barbelldensitycompletevbarbell} (left). 
\begin{figure}[h!]
\centering
\includegraphics[width=0.47\linewidth]{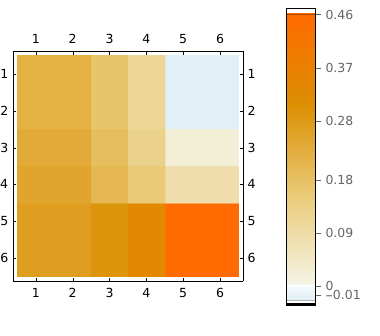}
\includegraphics[width=0.47\linewidth]{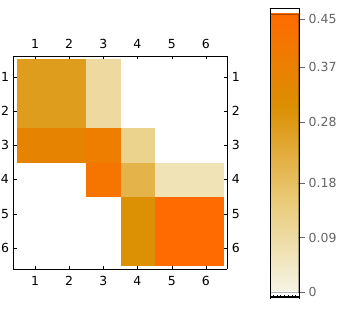}
\caption{A heat map visualization of the optimal $\vec{A}_*$ for a fully-connected network (left), and $\vec{A}_*$ for the barbell network (right). }
    \label{fig:barbelldensitycompletevbarbell}
\end{figure}

Observe that negative entries arise in the matrix locations corresponding to the risk exchange between the agents 
with mean $1$ and mean $64$ losses. Next, we restrict risk-sharing to the following barbell network
$$
\raisebox{-.5\height}{
\begin{tikzpicture}[scale=1, every node/.style={circle, draw, fill=white, inner sep=2pt}]
  \node (a1) at (0,0) {2};
  \node (a2) at (0,1) {1};
  \node (a3) at (1,0.5) {3};
  \draw (a1) -- (a2) -- (a3) -- (a1);

  \node (b1) at (3,0) {6};
  \node (b2) at (3,1) {5};
  \node (b3) at (2,0.5) {4};
  \draw (b1) -- (b2) -- (b3) -- (b1);

  \draw (a3) -- (b3);
\end{tikzpicture}} 
$$
and use Theorem \ref{MainResults} to compute the optimal risk-sharing matrix $\vec{A}_*$, where only friends in the barbell network share risk, see Figure \ref{fig:barbelldensitycompletevbarbell} (right). 
In this barbell network, agents are connected if their expected losses are within a factor of four of each other, which, in this case, eliminates the negative entries.
This example motivates further study both of nonnegativity conditions for network-based risk-sharing as well as processes for constructing risk-sharing networks based on the distribution of agents' losses, see \S \ref{sec:discussion} for further discussion.

\begin{remark}
As noted in Section \ref{sec:nonneg}, the non-negativity conditions established in Lemma 
\ref{MuConditions}, Proposition \ref{lem:gencase},  
Lemma \ref{FriendsEqualPos}, and Corollary \ref{cstarcovarianceeq} involve comparisons between statistical quantities of individuals with global counterparts.
The examples in Sections \ref{sec:disconnnect} and \ref{sec:barbell} show that network restrictions can eliminate these negative entries. Informally speaking, non-negativity is achieved by restricting risk sharing between agents whose expected losses differ substantially.

Our numerical examples suggest that when economic agents engage in a risk-sharing network with connections between agents with similar risk profiles, as quantified in Section \ref{sec:nonneg},
the negative entries are eliminated. Such a network organization may also have implications for individual incentives to participate in risk-sharing for risk-averse agents, see Section \ref{sec:discussion} for further discussion.
\end{remark}

\section{Proof of main result} \label{sec:proofmainresult}
This section is organized as follows. First, in \S \ref{sec:proofnotation}, 
we introduce notation. Second, in \S \ref{sec:proofthm}, we prove Theorem \ref{MainResults}.

\subsection{Notation} \label{sec:proofnotation}
We introduce some tensor algebra notation used in the proofs below. The tensor product $\vec{A} \otimes \vec{B}$ of two matrices $\vec{A}$ and $\vec{B}$ of respective sizes $m\times n$ and $ p\times q$ is the $mp\times nq$ matrix
$$\vec{A}\otimes \vec{B} = \begin{bmatrix}
    a_{11}\vec{B}&\dots & a_{1n}\vec{B}\\
    \vdots &\ddots & \vdots\\
    a_{m1}\vec{B}&\dots & a_{mn}\vec{B}
\end{bmatrix}.$$
Let $\vec{C}$ and $\vec{D}$ be matrices of dimensions $m\times n$ and $q\times n$, respectively. The (column) concatenation $\vec C\oplus \vec D$ is the $(m+q)\times n$ matrix
$$
\vec{C}\oplus \vec{D} = \begin{bmatrix}
    \vec{C}\\\vec{D}
\end{bmatrix}.
$$
Let $\vec{E} = (e_{i j})$ be an $m\times n$ matrix. The (column) vectorization of $\vec{E},$ denoted $\text{vec}(\vec E),$ is the $mn$-dimensional column vector 
$$
\mathrm{vec}(\vec{E}) = \begin{bmatrix}
    e_{11}& e_{21}& \cdots \; e_{m1}&e_{12}&\dots &e_{m2} & \dots & e_{1n} & \dots & e_{mn}
\end{bmatrix}^\top.
$$
\normalsize
If $\vec{F} = (f_{i j})$ and $\vec{G} = (g_{i j})$ are $m \times n$ matrices, then $\vec{F} \odot \vec{G}$ denotes the entrywise product whose $(i,j)$-th entry is
$$
(\vec{F} \odot \vec{G})_{i j} = f_{i j} g_{i j}.
$$

\subsection{Proof of Theorem \ref{MainResults}} \label{sec:proofthm}
The proof of Theorem \ref{MainResults} is divided into three steps. First, in Step \ref{step1}, we transform the optimization problem \eqref{opteq} into a quadratic program that only has equality constraints. Second, in Step \ref{step2}, we justify  that this quadratic program has a unique solution determined by the KKT conditions. Third, in Step \ref{step3}, we rewrite the KKT conditions for the transformed problem in the notation of the original optimization problem.

Recall that $\vec{X}$ is the nonnegative $n$-dimensional random vector with mean $\vec{\mu}$ and covariance matrix $\vec{\Sigma}$, and that $G = (V,E)$ is an undirected graph whose vertices $V = \{1,\ldots,n\}$ correspond to agents. Let $\vec{W}$ denote the adjacency matrix of $G$. Set 
$$
\vec{Z} = \vec{1} \vec{1}^\top - \vec{W} - \vec{I},
$$
to be an indicator for the absence of an edge. Then, the optimization problem \eqref{opteq} can be rewritten as
\begin{equation}
\label{opteqZ}
\begin{cases}
\text{minimize} & \frac{1}{2}\tr(\vec{A} \vec{\Sigma} \vec{A}^\top) \\
\text{subject to} & \vec{A} \vec{\mu}  = \vec{\mu} , \quad \vec{1}^\top \vec{A} = \vec{1}^\top, \quad 
\vec{A} \odot \vec{Z} = \vec{0},
\end{cases}
\end{equation}
where $\vec{A} \odot \vec{Z}$ enforces the constraint that only friends can share risk. The following result rewrites \eqref{opteqZ} as a quadratic program that only has equality constraints.

\begin{step} \label{step1}
Set
$$ 
\vec{x} :=\mathrm{vec}(\vec{A}),
\qquad 
\vec Q:=\vec{\Sigma}\otimes \vec{I}, \qquad \text{and} \qquad
\vec{c}:=\vec{\mu}\oplus \vec{1}_n\oplus \vec{0}_m,
$$
where $\vec{1}_n$ denotes an $n$-dimensional column vector of ones, and $\vec{0}_m$ denotes an $m$-dimensional column vector of zeros.
Define $\vec{B} := \vec{B}_{\vec{\mu}} \oplus \vec{B}_{\vec{1}} \oplus \vec{B}_{\vec{0}}$, where
$$
\vec{B}_{\vec{\mu}} := \left(\bigoplus_{i=1}^n \mu_i\vec{I}\right)^\top, \quad \vec{B}_{\vec{1}} := \bigoplus_{i=1}^n \left(\vec{0}_{(i-1)n}\oplus \vec{1}_n \oplus \vec0_{n^2-in} \right)^\top, \quad \vec{B}_{\vec{0}} := \bigoplus_{i=1}^m \vec{e}_{j_i}^\top,
$$
where $\vec{e}_i$ is the $i$-th standard basis vector of dimension $n^2$,  and
$j_1,\ldots,j_m$ are the indices of the nonzero entries of $\mathrm{vec}(\vec{Z})$. 
Then, the optimization problem
\begin{equation} \label{quadprogram}
\begin{cases}
\text{minimize} & \frac{1}{2}\vec{x}^\top\vec{Q} \vec{x} \\
\text{subject to} & \vec{Bx}=\vec{c}
\end{cases}
\end{equation}
is equivalent to \eqref{opteqZ} in the sense that if $\vec{x}_* = \mathrm{vec}(\vec{A}_*)$, then $\vec{x}_*$ is in the optimal set of \eqref{quadprogram} if and only if $\vec{A}_*$ is in the optimal set of \eqref{opteqZ}.
\end{step}

\begin{proof}[Proof of Step \ref{step1}]
First, we show that 
$\vec{x}^\top \vec{Q} \vec{x} = \tr( \vec{A} \vec{\Sigma} \vec{A}^\top).$
By the definition of $\vec{x}$ and $\vec{Q}$, we have
$$
\vec{x}^\top \vec{Q} \vec{x} = \sum_{k_1,k_2=1}^{n^2} \vectorize(\vec{A})_{k_1} (\vec{\Sigma} \otimes \vec{I})_{k_1 k_2} \vectorize(\vec{A})_{k_2}.
$$
By writing $k_1 = i_1 + n(j_1-1)$ and $k_2 = i_2 + n(j_2-1)$, it follows from the definition of vectorization and the tensor product that
$$
\sum_{k_1,k_2=1}^{n^2} \vectorize(\vec{A})_{k_1} (\vec{\Sigma} \otimes \vec{I})_{k_1 k_2} \vectorize(\vec{A})_{k_2} 
 = \sum_{i_1,j_1,i_2,j_2=1}^n  a_{i_1 j_1} \sigma_{j_1 j_2} \delta_{i_1 i_2}  a_{i_2 j_2},
$$
where $a_{i j}$, $\sigma_{i j}$, and $\delta_{i j}$ denote the entries of $\vec{A}$, $\vec{\Sigma}$, and $\vec{I}$, respectively. Using the fact that $\delta_{i j} = 1$ when $i=j$ and $\delta_{i j}=0$ otherwise gives
$$
\sum_{i_1,j_1,i_2,j_2=1}^n  a_{i_1 j_1} \sigma_{j_1 j_2} \delta_{i_1 i_2}  a_{i_2 j_2} = \sum_{i_1,j_1,j_2=1}^n  a_{i_1 j_1} \sigma_{j_1 j_2}  a_{i_1 j_2} = \tr( \vec{A} \vec{\Sigma} \vec{A}^\top),
 $$
 where the final equality follows from the definition of matrix multiplication and the trace.

Next, we will show the equivalence of the constraints.
Recall that 
$$
\vec{B} := \vec{B}_{\vec{\mu}} \oplus \vec{B}_{\vec{1}} \oplus \vec{B}_{\vec{0}}
\quad \text{and} \quad
\vec{c}=\vec{\mu}\oplus \vec{1}_n\oplus \vec{0}_m.
$$
First, we will show the constraint $\vec{B}_{\vec{\mu}} \vec{x} = \vec{\mu}$ is equivalent to $\vec{A} \vec{\mu} = \vec{\mu}$ by showing $\vec{B}_{\vec{\mu}} \vec{x} = \vec{A} \vec{\mu}$. Fix $i \in \{1,\ldots,n\}$. We have
$$
(\vec{B}_{\vec{\mu}} \vec{x})_i = \sum_{k_1=1}^{n^2} (\vec{B}_{\vec{\mu}})_{i,k_1} x_{k_1}.
$$
If we write $k_1 = i_1 + n(j_1 -1)$, then by the definition of $\vec{B}_{\vec{\mu}}$ we have
$\vec{B}_{\vec \mu} = [\mu_1 \vec I, \ldots, \mu_n \vec I].$ Thus,
$$
\sum_{k_1=1}^{n^2} (\vec{B}_{\vec{\mu}})_{i,k_1} x_{k_1} = \sum_{i_1,j_1=1}^n 
\delta_{{i i_1}} \mu_{{j_1}}
a_{i_1 j_1} 
$$
where $\delta_{i j},a_{i j}$ denote the entries of $\vec{I}, \vec{A}$, respectively. Using the fact that  $\delta_{i j} = 1$ if $i=j$ and $\delta_{i j} = 0$ otherwise gives
$$
\sum_{i_1,j_1=1}^n 
\delta_{{i i_1}} \mu_{{j_1}}
a_{i_1 j_1} 
= 
\sum_{{j_1}=1}^n a_{i {j_1}}  \mu_{{j_1}}   = (\vec{A} \vec{\mu})_i,
$$
which establishes the equivalence to the first constraint. Second, we show that $\vec{B}_{\vec{1}} \vec{x} = \vec{1}_n$ is equivalent to $\vec{1}_n^\top \vec{A} = \vec{1}_n^{{\top}}$ by showing $\vec{B}_{\vec{1}} \vec{x} = \vec{A}^\top \vec{1}_n$. Fix an index $i \in \{1,\ldots,n\}$. We have 
$$
(\vec{B}_{\vec{1}} \vec{x})_i = \sum_{k_1{=1}}^{n^2} \left( \vec{B}_{\vec{1}} \right)_{i,k_1} x_{k_1}.
$$
By writing $k_1 = i_1 + n(j_1 -1)$ and using the definition of  $\vec{B}_{\vec{1}}$, we have
$$
\sum_{k_1 {=1}}^{n^2} \left( \vec{B}_{\vec{1}} \right)_{i,k_1} x_{k_1} = \sum_{i_1,{j_1}=1}^n \delta_{i j_1} a_{i_1 j_1},
$$
where $\delta_{i j} = 1$ if $i=j$ and $\delta_{i j}=0$ otherwise, and $a_{i j}$ are the entries of $\vec{A}$. Since
$$
\sum_{i_1,{j_1}=1}^n \delta_{i j_1} a_{i_1 j_1} = \sum_{i_1=1}^n a_{i_1 i} = (\vec{A}^\top \vec{1})_i,
$$
the equivalence of the second constraint is established. Finally, we show that
$\vec{B}_{\vec{0}} \vec{x} = \vec{0}_m$ is equivalent to $\vec{A} \odot \vec{Z} = \vec{0}$. 
The entrywise product in the final constraint $\vec{A} \odot \vec{Z} = \vec{0}$ can be directly vectorized as
$$
\vectorize(\vec{Z}) \odot \vectorize(\vec{A}) = \vectorize(\vec{Z}) \odot \vec{x} = \vectorize(\vec{0}).
$$
Recall that $j_1,\ldots,j_m$ are the indices of the nonzero entries of $\vec{Z}$, so 
$\vectorize(\vec{Z}) \odot \vec{x} = \vectorize(\vec{0})$ is equivalent to
$x_{j_i} = 0$ for $i \in \{1,\ldots,m\}$. Since $(\vec{B}_{\vec{0}} \vec{x})_i = x_{j_i}$, the equivalence of the final constraint is established, which completes the proof.
\end{proof}

\begin{step}\label{step2}
In addition to the notation introduced in Step \ref{step1}, assume that $\vec{\Sigma}$ is positive definite. Then, the quadratic program 
\begin{equation} \label{eq:minproblemquad}
\begin{cases}
\text{minimize} & \frac{1}{2}\vec{x}^\top\vec{Q}\vec{x} \\
\text{subject to} & \vec{Bx}=\vec{c}
\end{cases}
\end{equation}
has a unique solution $\vec{x}_*$ characterized by the KKT conditions
\begin{equation} \label{eq:kkt}
\begin{bmatrix}
\vec{Q} & \vec{B}^\top \\
\vec{B} & \vec{0}
\end{bmatrix}
\begin{bmatrix}
\vec{x}_* \\
\vec{\nu}_*
\end{bmatrix}
= \begin{bmatrix}
\vec{0} \\
\vec{c}
\end{bmatrix},
\end{equation}
where $\vec{\nu}_*$ are Lagrange multipliers that exist but need not be unique.
\end{step}

\begin{proof}[Proof of Step \ref{step2}] 
First, we argue that this optimization has a unique solution. Since the admissible set of points that satisfy the constraints contains $\vectorize(\vec{I})$, the optimal set is nonempty. Recall that if the objective function in a convex optimization problem is strictly convex, the optimal set contains at most one point, see \cite[Section 4.2]{boyd2004convex}.
Since $\vec\Sigma$ and $\space \vec{I}$ are both positive definite, their tensor product $\vec Q = \vec\Sigma \otimes \vec{I}$ is positive definite since the eigenvalues of 
$\vec\Sigma \otimes \vec{I}$ are products of the eigenvalues of $\vec{\Sigma}$ and $\vec{I}$. It follows that $\frac{1}{2} \vec{x}^\top \vec{Q} \vec{x}$ is strictly convex, and thus \eqref{eq:minproblemquad} has a unique solution.

Second, we justify that the KKT conditions 
\eqref{eq:kkt} characterize the solution. The fact that the KKT conditions for the optimization can be written as \eqref{eq:kkt} follows from \cite[Example 5.1]{boyd2004convex}. The objective function is convex and differentiable, and the problem does not have inequality constraints. Therefore, the KKT conditions are necessary and sufficient conditions for $(\vec{x}_*, \vec{\nu}_*)$ to be primal and dual optimal, with zero duality gap, see \cite[Page 244]{boyd2004convex}.
Since the KKT conditions are necessary and sufficient conditions, Lagrange multipliers $\vec{\nu}_*$ such that $(\vec{x}_*,\vec{\nu}_*)$ satisfies \eqref{eq:kkt} must exist, but the Lagrange multipliers need not be unique.
\end{proof}

Finally, to complete the proof of Theorem \ref{MainResults}, we write the KKT conditions \eqref{eq:kkt}  of the transformed problem using the notation of the original optimization problem.

\begin{step}\label{step3}
The optimization problem
\begin{equation}
\label{opteqZ2}
\begin{cases}
\text{minimize} & \frac{1}{2}\tr(\vec{A} \vec{\Sigma} \vec{A}^\top) \\
\text{subject to} & \vec{A} \vec{\mu}  = \vec{\mu} , \quad \vec{1}^\top \vec{A} = \vec{1}^\top, \quad 
\vec{A} \odot \vec{Z} = \vec{0},
\end{cases}
\end{equation}
has a unique solution 
$$
 \vec{A}_* = \frac{1}{n} \vec{1 1}^\top  + \left(\vec I - \frac{1}{n} \vec{1 1}^\top \right)
\left( \frac{1}{a} \vec{\mu \mu}^\top + \vec{\Gamma} \left(
\frac{1}{a} \vec\Sigma^{-1}\vec{\mu \mu}^\top
 - \vec I \right) \right) \vec\Sigma^{-1},
$$ 
where $a = \vec\mu^\top \vec\Sigma^{-1}\vec\mu$ and $\vec{\Gamma} = (\gamma_{ij}) \in \mathbb{R}^{n \times n}$ 
satisfies $\gamma_{ij} = 0$ when $i=j$ or $\{i,j\} \in E$, and the other entries $\gamma_{ij}$ 
are chosen so that the equations
\begin{equation*} 
 \left( 
  \frac{1}{n} \vec{1 1}^\top   +
 \left(\vec I - \frac{1}{n} \vec{1 1}^\top \right)
 \left(
   \frac{1}{a} \vec{\mu \mu}^\top +
 \vec{\Gamma}
\left(\frac{1}{a} \vec\Sigma^{-1}\vec{\mu \mu}^\top
 - \vec I \right)\right)   \vec\Sigma^{-1} 
 \right)_{ij}
 = 0
 \end{equation*}
are satisfied for all $i,j$ such that $\{i,j\} \in \overline{E}$. Such a matrix $\vec{\Gamma}$ always exists but need not be unique.
\end{step}

\begin{proof}[Proof of Step \ref{step3}]
The KKT conditions \eqref{eq:kkt} from Step \ref{step2} can be written as
$$
\vec{Q} \vec{x}_* + \vec{B}^\top \vec{\nu}_* = \vec{0}, \quad \text{and} \quad \vec{B} \vec{x}_* = \vec{c} {,}
$$
where $\vec{x}_*$ is the unique solution to \eqref{eq:minproblemquad}, and 
$\vec{\nu}_*$ are Lagrange multipliers that exist but need not be unique.
Let $\vec{A}_*$ be the $n \times n$ matrix such that $\vec{x}_* = \vectorize( \vec{A}_*)$. Previously, in the Proof of Step \ref{step1}, we established the equivalence of the constraint $\vec{B} \vec{x}_* = \vec{c}$ to the three constraints:
$$
\vec{A}_{*} \vec{\mu}  = \vec{\mu} , \quad \vec{1}^\top \vec{A}_{*} = \vec{1}^\top, \quad 
\vec{A}_{*} \odot \vec{Z} = \vec{0}.
$$
Next, we rewrite the equation $\vec{Q} \vec{x}_* + \vec{B}^\top \vec{\nu}_* = \vec{0}$ in terms of $\vec{A}_{*}$.  Next, introduce a convenient notation for the Lagrange multipliers $\vec{\nu}_*$
$$
\vec{\nu}_* = (-\vec{\beta}) \oplus (-\vec{\lambda}) \oplus \vec{\gamma},
$$
where $\vec{\beta} \in \mathbb{R}^n$, $\vec{\lambda} \in \mathbb{R}^n$, and $\vec{\gamma} \in \mathbb{R}^m$, where $m$ is the number of nonzero entries of $\vec{Z}$. Fix $k_1 \in \{1,\ldots,n^2\}$. If we write $k_1 = i_1 + n(j_1 -1)$ for $i_1,j_1 \in \{1,\ldots,n\}$, then by the definition of $\vec{Q}$ and $\vec{B}$,
$$
(\vec{Q} \vec{x}_*)_{k_1} + (\vec{B}^\top \vec{\nu}_*)_{k_1} = 
\left(\vec{A}_{*} \vec{\Sigma}-\vec{1\lambda}^\top-\vec{\beta\mu}^\top + \vec{\Gamma} \right)_{i_1 j_1},
$$
where $\vec{\Gamma} = (\gamma_{ij}) \in \mathbb{R}^{n \times n}$ is defined as follows: we have $\gamma_{i j} = 0$ when $i=j$ or $\{i,j\} \in E$, and the other $m$ entries of $\vec{\Gamma}$ each correspond to one of the $m$ entries of $\vec{\gamma}$.
With this notation, it follows that $\vec{Q} \vec{x}_* + \vec{B}^\top \vec{\nu}_* = \vec{0}$
is equivalent to 
$$
\vec{A}_{*} \vec{\Sigma}-\vec{1\lambda}^\top-\vec{\beta\mu}^\top + \vec\Gamma = \vec{0}.
$$
Solving for $\vec{A}_{*}$ gives
\begin{equation} \label{eqA1}
\vec{A}_{*} = (\vec{1\lambda}^\top+\vec{\beta\mu}^\top - \vec\Gamma)\vec\Sigma^{-1}.
\end{equation}
Using the constraint $\vec1^\top \vec{A}_{*} = \vec1^\top$ and solving for $\vec\lambda^\top$ yields
$$
\vec\lambda^\top = \frac{1}{n}(\vec1^\top\vec\Sigma-\vec1^\top\vec{\beta\mu}^\top + \vec1^\top \vec{\Gamma}).
$$
Substituting this formula for $\vec\lambda^\top$ into \eqref{eqA1}, we obtain
\begin{equation} \label{eqA2}
\vec{A}_{*} = \frac{1}{n} \vec{1 1}^\top  + \left(\vec I - \frac{1}{n} \vec{1 1}^\top \right)\vec{\beta \mu}^\top\vec\Sigma^{-1} -  \left(\vec I - \frac{1}{n} \vec{1 1}^\top \right) \vec\Gamma\vec\Sigma^{-1}.
\end{equation}
Using the constraint $\vec{A}_{*} \vec{\mu} =  \vec\mu$ and solving for $\vec{\beta}$ gives
$$
\vec\beta = (\vec\mu+\vec\Gamma \vec\Sigma^{-1}\vec\mu)(\vec\mu^\top\vec\Sigma^{-1}\vec\mu)^{-1}+ c \vec1,
$$
where $c$ is some scalar. If we define $a := \vec\mu^\top\vec\Sigma^{-1}\vec\mu$, then
$$
\vec\beta = \frac{1}{a} (\vec\mu+\vec\Gamma \vec\Sigma^{-1}\vec\mu)+ c \vec1.
$$
Substituting this formula for $\vec\beta$ in \eqref{eqA2} gives
\begin{equation}\label{eqA3}
  \vec{A}_{*} = \frac{1}{n} \vec{1 1}^\top  + \left(\vec I - \frac{1}{n} \vec{1 1}^\top \right)
\left( \frac{1}{a} \vec{\mu \mu}^\top + \vec\Gamma \left(
\frac{1}{a} \vec\Sigma^{-1}\vec{\mu \mu}^\top
 - \vec I \right) \right) \vec\Sigma^{-1}.  
\end{equation}
Using the constraint $\vec Z \odot \vec{A}_{*} = \vec0$, we observe that
\begin{equation}\label{eqZdot}
\vec Z \odot \left( \frac{1}{n} \vec{1 1}^\top  + \left(\vec I - \frac{1}{n} \vec{1 1}^\top \right)
\left( \frac{1}{a} \vec{\mu \mu}^\top + \vec\Gamma \left(
\frac{1}{a} \vec\Sigma^{-1}\vec{\mu \mu}^\top
 - \vec I \right) \right) \vec\Sigma^{-1} \right) = \vec0.
\end{equation}
 Using the fact that 
 $\vec{Z} = \vec{1} \vec{1}^\top - \vec{W} - \vec{I}$
is an indicator function for the absence of an edge, 
we can write out the equations \eqref{eqZdot} explicitly as
\begin{equation} \label{eq:linearsystem2}
 \left( 
 \left(\vec I - \frac{1}{n} \vec{1 1}^\top \right)
 \left(\vec{\Gamma}
\left(\frac{1}{a} \vec\Sigma^{-1}\vec{\mu \mu}^\top
 - \vec I \right)+  \frac{1}{a} \vec{\mu \mu}^\top\right)   \vec\Sigma^{-1} 
+ \frac{1}{n} \vec{1 1}^\top   
 \right)_{i j}
 = 0
\end{equation}
for all $i \not = j$ such that $\{i,j\} \not \in E$.

\end{proof}

\begin{remark}[Computation of $\vec{\Gamma}$] \label{rmk:computingAstart} 
The linear system of equations \eqref{eq:linearsystem2} that determines the nonzero entries of $\vec{\Gamma} \in \mathbb{R}^{n \times n}$ consists of selected equations of a linear system of the form $\vec{E} \vec{\Gamma} \vec{F} = \vec{G}$ where $\vec{E}, \vec{F}, \vec{G} \in \mathbb{R}^{n \times n}$ are given matrices. 
We can write this linear system in matrix-vector form using tensor algebra notation. In particular, we have
$$
\vec{E} \vec{\Gamma} \vec{F} = \vec{G}
\quad
\iff
\quad
(\vec{F}^\top \otimes \vec{E}) \vectorize(\vec{\Gamma}) = \vectorize(\vec{G}).
$$
By selecting an $m \times m$ principal submatrix of $\vec{F}^\top \otimes \vec{E}$ and the corresponding $m$ elements of $\vectorize(\vec{G})$, where $m$ is the number of ordered pairs $(i,j)$ such that $\{i,j\} \in \overline{E}$, one arrives at a linear system in standard matrix vector form, which can be directly solved using standard methods. Finally, we note that this direct approach to computing $\vec{\Gamma}$ involves constructing an $n^2 \times n^2$ matrix $\vec{F}^\top \otimes \vec{E}$, which may be prohibitive for large $n$. In such cases, the linear system could be solved iteratively instead.
\end{remark}

\section{Summary and Discussion}
\label{sec:discussion}

In this paper, we consider optimal linear actuarially fair P2P risk-sharing on any connected graph without additional restrictions. When a complete graph is considered, our results agree with prior results of Feng, Liu, and Taylor \cite{feng2023peer}. However, the results presented in this paper also apply to arbitrary connected graphs, which enables the restriction that only friends can share risk.

An important feature of the present framework is the use of variance as the risk measure. Variance leads to a quadratic objective function, which makes it possible to derive explicit solutions in Theorems \ref{MainResults} and \ref{thmoptshare}. Other risk measures may be more appropriate in some applications, but they generally lead to different optimization problems. 
In that case, one would not expect closed-form solutions similar to those obtained here, although numerical optimization techniques could still be used.

We emphasize that the optimization problems considered in this paper seek to minimize \emph{aggregate} variance, which does not, in general, imply variance reduction for every individual.
In certain special cases, individual variance reduction does occur; for example, in the i.i.d. complete-graph case, both formulations reduce to ordinary equal pooling, and each agent's variance is reduced from $\sigma^2$ to $\sigma^2/n$. However, this conclusion does not hold automatically for general network structures. Consequently, the resulting allocations should not be interpreted as automatically satisfying individual rationality.

Considering risk sharing rules on networks that impose individual rationality constraints is a natural direction for future work. In practice, agents may require additional conditions before agreeing to participate in a risk-sharing arrangement. For example,  given a fixed risk sharing network structure, one could require that each agent's post-sharing risk does not exceed its pre-sharing risk, or consider other notions of fairness. Another possibility is to consider methods of constructing risk sharing networks that promote a willingness to participate. Incorporating such requirements would lead to a different optimization problem and is left for future work.

We further examine risk-sharing rules with the additional condition that friends of agent $i$ take an equal fractional share of the risk of agent $i$. In this case, the optimal linear risk-sharing rule is related to the graph Laplacian $\vec L$.

We also identify necessary and sufficient conditions for nonnegativity of the entries of the optimal linear risk-sharing matrix for certain cases. For complete networks with no additional restrictions, the mean vector $\vec\mu$ and covariance matrix $\vec\Sigma$ determine nonnegativity, while for any rules that enforce the friends share equal risk assumption, conditions are in terms of the degree of each node in the network, the mean vector, and the covariance matrix.

The theoretical results are illustrated through several examples that demonstrate the versatility and utility of these new approaches to linear P2P risk-sharing rules. Costs and benefits, as measured by the sum of variances after risk-sharing, are discussed.  Examples illustrate that, although negative entries may appear in the optimal linear rule over complete graphs, restricting the optimization by disconnecting nodes associated with negative entries can result in exclusively nonnegative entries. The provided examples and analysis establish several options for risk-sharing networks depending on which properties, such as nonnegativity, equal risk shared between friends, or the lowest possible overall variance after risk-sharing, would be most beneficial to the specific risk pool.

The results in this paper also motivate other questions for further study. First, it may be interesting to consider the connection between network structure and the willingness of agents to join a P2P insurance network. For example, participants may be more willing to join a risk-sharing network if they only share risk with other agents whose losses have expectation and variance within a certain range of their own. The barbell example in Section \ref{sec:barbell} demonstrates a network constructed based on mitigating risk-sharing between agents with extreme differences in expectation, although different assumptions or guidelines for establishing connections in a risk-sharing network may produce different results or require new assumptions to preserve willingness to join. One consideration of interest is preferential attachment, where economic parameters of the agents outside of expectation alone affect how a network is formed. It may also be of interest to extend our network risk-sharing models to consider multiple periods or continuous time rather than the static models examined in this paper.

\bibliographystyle{plain}
\bibliography{refs}

\end{document}